\documentclass[11pt]{amsart}
\usepackage{amscd}
\usepackage{amsmath}
\usepackage{amssymb}
\usepackage{latexsym}
\usepackage{xypic}

\input
xy \xyoption{all}

\newtheorem{thm}{Theorem}[section]
\newtheorem{prop}[thm]{Proposition}
\newtheorem{lem}[thm]{Lemma}

\newtheorem{lem-def}[thm]{Lemma-Definition}
\newtheorem{cor}[thm]{Corollary}
\theoremstyle{remark}
\newtheorem{ex}[thm]{Example}
\newtheorem{rmk}{Remark}[section]
\theoremstyle{definition}
\newtheorem{dfn}{Definition}[section]

\numberwithin{equation}{section}

\newcommand{\bbK}{{\mathbb K}}
\newcommand{\bbL}{{\mathbb L}}

\newcommand{\bbV}{{\mathbb V}}
\newcommand{\bbZ}{{\mathbb Z}}

\newcommand{\bbN}{{\mathbb N}}
\newcommand{\da}{{\mathbb A}}

\newcommand{\calC}{{\mathcal C}}
\newcommand{\calD}{{\mathcal D}}
\newcommand{\calE}{{\mathcal E}}

\newcommand{\calK}{{\mathcal K}}
\newcommand{\calL}{{\mathcal L}}

\newcommand{\calO}{{\mathcal O}}
\newcommand{\calP}{{\mathcal P}}

\newcommand{\calT}{{\mathcal T}}

\newcommand{\calV}{{\mathcal V}}

\newcommand{\calZ}{{\mathcal Z}}
\newcommand{\ff}{{\mathcal F}}
\newcommand{\g}{{\mathcal G}}
\newcommand{\h}{{\mathcal H}}
\newcommand{\nc}{\newcommand}

\newcommand{\oo}{{\mathcal O}}
\newcommand{\ad}{{\mathcal A}}

\nc{\on}{\operatorname}

\newcommand{\Aut}{{\on{Aut}}}
\newcommand{\End}{{\mathrm{End}}}
\newcommand{\Ext}{{\mathrm{Ext}}}

\newcommand{\Hom}{{\mathrm{Hom}}}
\nc{\id}{{\on{id}}}

\newcommand{\spec}{{\on{Spec}}}
\newcommand{\Frac}{{\on{Frac}}}
\nc{\Tate}{{\on{Tate}}}

\newcommand{\GL}{{\on{GL}}}

\nc{\Comm}{{\on{Comm}}}
\nc{\Det}{{\calD et}}

\newcommand{\lrto}{\longrightarrow}

\begin{document}

\author{Denis Osipov, Xinwen
Zhu}

\title{A categorical proof of the Parshin reciprocity laws on algebraic
surfaces}

\maketitle

\begin{abstract} We define and study the $2$-category of torsors over a
Picard groupoid, a central extension of a group by a Picard
groupoid, and  commutator maps in this central extension. Using this
in the context of two-dimensional local fields and two-dimensional
ad\`ele theory we obtain the two-dimensional tame symbol and a new
proof of Parshin reciprocity laws on an algebraic surface.
\end{abstract}

\section{Introduction}

Let $C$ be a projective algebraic curve over a perfect field $k$.
Then there is the following famous Weil reciprocity law:
\begin{equation} \label{prW}
\prod_{p \in C} \on{Nm}_{k(p)/k}\{f,g \}_p =1 \mbox{,}
\end{equation}
where $f,g \in k(C)^{\times}$, $\{f,g \}_p$ is the
one-dimensional tame symbol, which is equal to $(-1)^{\nu_p(f)
\nu_p(g)} \frac{f^{\nu_p(g)}}{g^{\nu_p(f)}}(p)$, and $k(p)$ is
the residue field of the point $p$. The product~\eqref{prW} contains
only finitely many non-equal to $1$ terms.

There is the proof of Weil reciprocity law (and the analogous
reciprocity law for residues of rational differential forms: sum of
residues equals to zero) by reduction to the case of ${\bf{P}}^1_k$
using the connection between tame symbols (and residues of
differentials) in extensions of local fields, see, for example,
\cite[ch.~2-3]{S}.

On the other hand, Tate gave in~\cite{T} the definition of local
residue of differential form as some trace of an
infinite-dimensional matrix. Starting from this definition he gave
an intrinsic  proof of the residue formula on a projective algebraic
curve $C$ using the fact that $\dim_k H^i (C, \oo_C) < \infty$, $i
=0,1$.

The multiplicative analog of Tate's approach, i.e. the case of the
tame symbol and the proof of Weil reciprocity law, was done later by
Arbarello, De Concini and Kac in~\cite{ACK}. They used the central
extension of some infinite dimensional group of matrices $\GL(K)$ by
the group $k^{\times}$ such that the group $\GL(K)$ acts on the
field $K=k((t))$, and obtained the tame symbol up to sign as the
commutator of the lifting of two elements from $K^{\times} \subset
\GL(K)$ to this central extension. Hence, as in Tate's proof
mentioned above, they obtained an intrinsic proof of the Weil
reciprocity law on an algebraic curve. However, in this proof the
exterior algebra of finite-dimensional $k$-vector spaces was used.
Therefore difficult sign conventions were used in this paper to
obtain the reciprocity law. To avoid these difficulties,
in~\cite{BBE} Beilinson, Bloch and Esnault used the category of
graded lines instead of the category of lines. The category of
graded lines has non-trivial commutativity constraints multipliers
$(-1)^{mn}$, where $m, n \in \bbZ$ are corresponding gradings. In
other words, they used the Picard groupoid of graded lines which is
a non-strictly commutative instead of strictly commutative Picard
groupoid. It was the first application of this notion of
non-strictly commutative Picard groupoid.

\medskip

Now let $X$ be an algebraic surface over a perfect field $k$. For
any pair $x \in C$, where $C \subset X$ is a curve that $x\in C$ is
a closed point, it is possible to define the ring $K_{x,C}$ such
that $K_{x,C}$ is isomorphic to the two-dimensional local field
$k(x)((t))((s))$ when $x$ is a smooth point on $C$ and $X$. If $x$
is not a smooth point, then $K_{x,C}$ is a finite direct sum of
two-dimensional local fields (see section~\ref{sec5.2} of this
paper). For any two-dimensional local field $k'((t))((s))$ one can
define the two-dimensional tame symbol of $3$ variables with values
in $k'^{\times}$, see section~\ref{ts} and~\cite{Pa0}, \cite[\S
3]{Pa}. Parshin formulated and proved the reciprocity laws for
two-dimensional tame symbols, but his proof was never published.
Contrary to the $1$-dimensional case, there are a lot of reciprocity
laws for two-dimensional tame symbols, which belong to two types.
For the first type we fix a point on the surface and will vary
irreducible curves containing this point. For the second type we fix
a projective irreducible curve on the surface and will vary points
on this curve. Parshin's idea for the proof, for example, of more
unexpected first type of reciprocity laws, was to use the chain of
successive blowups  of points on algebraic surfaces. Later, Kato
generalized the reciprocity laws for excellent schemes by using the
reduction to the reciprocity law of Bass and Tate for Milnor
K-groups of some field $L(t)$, see~\cite[prop.~1]{K}. He used them
to construct an analog of the Gersten-Quillen complex for Milnor
K-theory.

In this paper, we give a generalization of Tate's proof of the
reciprocity law on an algebraic curve to the case of two-dimensional
tame symbols and obtain an intrinsic proof of Parshin reciprocity
laws for two-dimensional tame symbols on an algebraic surface.

To fulfill this goal, we first generalize the notion of a central
extension of a group by a commutative group and of the commutator
map associated to the central extension. More precisely, we define
and study in some detail the properties of the category of central
extensions of a group $G$ by a (non strictly commutative) Picard
groupoid $\calP$. Roughly speaking, an object in this category is a
rule to assign every $g\in G$ a $\calP$-torsor, satisfying certain
properties. For such a central extension $\calL$ we define a map
$C_3^{\calL}$ which is an analog of the commutator map. In this case
when $G$ is abelian, this commutator map is an anti-symmetric and
tri-multiplicative map from $G^3$ to the group $\pi_1(\calP)$. Let
us remark that to obtain some of these properties, we used the
results of Breen from~\cite{Br2} on group-like monoidal
$2$-groupoids. We hope these constructions would be of some
independent interest.

We then apply this formalism to $\calP=\calP ic^\bbZ$, where $\calP
ic^\bbZ$ stands for the Picard groupoid of graded lines. The key
ingredient here is Kapranov's graded-determinantal theory
from~\cite{Kap}, which associates a $\calP ic^\bbZ$-torsor to every
1-Tate vector space (a.k.a. locally linearly compact vector space).
This allows one to construct the central extension $\calD et$ of
$\GL(\bbK)$ by $\calP ic^\bbZ$, where $\bbK$ is a two-dimensional
local field (or more generally, a 2-Tate vector space). It turns out
that the two-dimensional tame symbol coincides with the commutator
map $C_3^{\calD et}$. Finally, using "semilocal" ad\`ele complexes
on an algebraic surface we obtain that the corresponding central
extension constructed by semilocal fields on the surface is the
trivial one. This leads us to a new proof of Parshin's reciprocity
laws on an algebraic surface, which is distinct from both Parshin's
original approach as well as Kato's.

Our approach to the reciprocity laws on the algebraic surfaces has
the following features. First, we use the non-strictly commutative
Picard groupoid, which can be regarded as another application of
this notion after \cite{BBE}. However, unlike the one-dimensional
case where one can just plays with the usual Picard groupoid of
lines (though complicated, as done in \cite{ACK}), the use of $\calP
ic^\bbZ$ is essential here. This indicates that the non-strictly
commutative Picard groupoid is an important and fruitful
mathematical object that deserves further attention. Also, in order
to apply this notion, we develop certain constructions in higher
categories (e.g. the commutator map $C_3^\calL$), which could be
potentially useful elsewhere. Second, as in the one-dimensional
case, our approach uses certain local-to-global (in other words,
factorization) principle. Since the local-to-global (factorization)
principle in the one-dimensional story is very important in the
Langlands program and conformal field theory, we hope our approach
is just a shadow of a whole fascinating yet explored realm of
mathematics. Finally, our approach can be generalized by replacing
the ground field $k$ by an Artinian ring $A$ (and even more general
rings) and we can obtain the reciprocity laws for two-dimensional
Contou-Carr\`{e}re symbols. By choosing appropriate $A$, this
specializes to residue formulas for algebraic surfaces\footnote{The
generalization of Tate's approach to the $n$-dimensional residue of
differential form was done in~\cite{Be1}, but this note contains no
proofs.}. We will carefully discuss this in the next paper.

\medskip

The paper is organized as follows. In section~\ref{gennon} we
describe some categorical constructions, which we need further on.
In section~\ref{picgr} we recall the definition of a Picard
groupoid. In section~\ref{strict vs non-strict} we discuss the
difference between strictly commutative and non-strictly commutative
Picard groupoids. In section~\ref{Ptorsors} we describe the
$2$-category of $\calP$-torsors, where $\calP$ is a Picard groupoid.
In section~\ref{hompic} we study the Picard groupoid of
homomorphisms from a group $G$ to a Picard groupoid $\calP$ and
describe the "commutator" of two commuting elements from $G$ with
values in $\pi_1(\calP)$. In section~\ref{sec2} we define and study
the Picard $2$-groupoid of central extensions of a group $G$ by a
Picard groupoid $\calP$. We define and study properties of the
commutator category of such a central extension, and finally study
the "commutator" of three commuting elements form $G$ with values in
$\pi_1(\calP)$. This section may be of independent interest.

In section~\ref{tate} we recall the theory of graded-determinantal
theories on Tate vector spaces. We recall the definition and basic
properties of the category of $n$-Tate vector spaces in
section~\ref{tvecs}. In section~\ref{detth} we recall the definition
of determinant functor from the exact category $(\Tate_0,
\on{isom})$ to the Picard groupoid $\calP ic^{\bbZ}$ of graded lines
and the definition of graded-determinantal theory on the exact
category $\Tate_1$ of $1$-Tate vector spaces.

In section~\ref{appl} we apply the constructions given above to
one-dimensional and two-dimensional local fields. In
section~\ref{ts} we review one-dimensional and two-di\-men\-si\-onal
tame symbols. In section~\ref{one-story} we obtain a description of
the one-dimensional (usual) tame symbol as some commutator. In
section~\ref{2-story} we obtain the two-dimensional tame symbol as
commutator of $3$ elements in some central extension of the group
$K^{\times}= k((t))((s))^{\times}$ by the Picard groupoid $\calP
ic^{\bbZ}$.

In section~\ref{reslaws} we obtain the reciprocity laws. In
section~\ref{Weil} we give the proof of Weil reciprocity law using
the constructions given above and ad\`ele complexes on a curve. In
section~\ref{sec5.2} we apply  the previous results in order  to
obtain a proof of Parshin's reciprocity laws on an algebraic surface
using "semilocal" ad\`ele complexes on an algebraic surface.

\medskip

\noindent {\bf Acknowledgments.} The first author is grateful to
L.~Breen, M.~Kapranov and A.N.~Parshin for useful discussions on
categorical approach to the two-dimensional tame symbol at various
periods of time. The first author was financially supported by
Russian Foundation for Basic Research (grant no.~08-01-00095) and by
the Programme for the Support of Leading Scientific Schools of the
Russian Federation (grant no.~NSh-4713.2010.1). The second author
thanks V. Kac for useful discussions. The research of the second
author is supported by NSF grant under DMS-1001280.

\section{General  nonsense} \label{gennon}

\subsection{Picard groupoid} \label{picgr}
Let $\calP$ be a Picard groupoid, i.e. a symmetric monoidal
group-like groupoid. Let us recall that this means that $\calP$ is a
groupoid, together with a bifunctor
\[+:\calP\times\calP\to\calP\]
and natural (functorial) isomorphisms
\[a_{x,y,z}: (x+y)+z\simeq x+(y+z), \]
called the associativity constraints, and natural (functorial) isomorphisms
\[c_{x,y}:x+y\simeq y+x,\]
called the commutativity constraints, such that:
\begin{enumerate}
\item[(i)] For each $x\in\calP$, the functor $y\mapsto x+y$ is an
equivalence.
\item[(ii)] The pentagon axiom holds, i.e. the following diagram
is commutative
\begin{equation}\label{pentagon}\xymatrix{&(x+y)+(z+w)\ar[dl]\ar[dr]&\\
x+(y+(z+w))\ar[d]&&((x+y)+z)+w\ar[d]\\
x+((y+z)+w)\ar[rr]&&(x+(y+z))+w}\end{equation}
\item[(iii)] The hexagon axiom holds, i.e. the following diagram
is commutative
\begin{equation}\label{hexagon}\xymatrix{&(x+y)+z\ar[dl]\ar[dr]& \\
(y+x)+z\ar[d]&&x+(y+z)\ar[d]\\
y+(x+z)\ar[dr]&&x+(z+y)\ar[dl]\\
&(x+z)+y&}\end{equation}
\item[(iv)] For any $x,y\in\calP$, $c_{y,x}c_{x,y}=id_{x+y}$.
\end{enumerate}

A unit $(e,\varphi)$ of $\calP$ is an object $e\in\calP$ together
with an isomorphism $\varphi:e+e\simeq e$. It is an exercise to show
that $(e,\varphi)$ exists and is unique up to a unique isomorphism.
For any $x\in\calP$, there is a unique isomorphism $e+x\simeq x$
such that the following diagram is commutative
\[\xymatrix{(e+e)+x\ar[rr]\ar[dr]&&e+(e+x)\ar[dl]\\
                   &e+x&}\]
and therefore $x+e\simeq e+x\simeq x$. For any $x\in\calP$, we choose an object, denoted by $-x$, together
with an isomorphism $\phi_x:x+(-x)\simeq e$. The pair $(-x,\phi_x)$
is called an inverse of $x$, and it is unique up to a unique
isomorphism. We choose for each $x$ its inverse $(-x,\phi_x)$, then we have a canonical isomorphism
\begin{equation}\label{inverse}
-(-x)\simeq e+(-(-x))\simeq (x+(-x))+(-(-x))\simeq x+((-x)+(-(-x)))\simeq x+e\simeq x,
\end{equation}
and therefore a canonical isomorphism
\begin{equation}\label{adjun}
(-x)+x\simeq (-x)+(-(-x))\simeq e.
\end{equation}
Observe that we have another isomorphism $(-x)+x\simeq x+(-x)\simeq e$ using the commutativity constraint. When the Picard groupoid $\calP$ is strictly commutative (cf. \S \ref{strict vs non-strict}), these
 two isomorphisms are the same (cf. \cite[lemma 1.6]{Zhu}),  but in general they are different.

If $\calP_1,\calP_2$ are two Picard groupoids, then $\Hom(\calP_1,\calP_2)$ is defined as follows.
Objects are $1$-homomorphisms, i.e., functors $F:\calP_1\to\calP_2$ together with isomorphisms $F(x+y)\simeq F(x)+F(y)$
such that the following diagrams are commutative:
$$
\xymatrix{F((x+y)+z)\ar[r]\ar[d]&(F(x)+F(y))+F(z)\ar[d]\\
F(x+(y+z))\ar[r]&F(x)+(F(y)+F(z)) \mbox{,}
}
$$
\begin{equation} \label{dz}
\xymatrix{
F(x+y)\ar[r]\ar[d]&F(x)+F(y)\ar[d]\\
F(y+x)\ar[r]&F(y)+F(x) \mbox{.}
}
\end{equation}

Morphisms in $\Hom(\calP_1,\calP_2)$ are $2$-isomorphisms, i.e.,
natural transformations
\[
\theta:F_1\to F_2
\]
such that the following diagram is commutative
\[\xymatrix{F_1(x+y)\ar[r]\ar_\theta[d]&F_1(x)+F_1(y)\ar^\theta[d]\\
F_2(x+y)\ar[r]&F_2(x)+F_2(y) \mbox{.}}\]

It is clear that $\Hom(\calP_1,\calP_2)$ has a natural structure as
a Picard groupoid. Namely,
\[(F_1+F_2)(x):=F_1(x)+F_2(x),\]
and the
isomorphism $(F_1+F_2)(x+y)\simeq (F_1+F_2)(x)+(F_1+F_2)(y)$ is the
unique one such that the following diagram is commutative
\[\xymatrix{F_1(x+y)+F_2(x+y)\ar[r]\ar[d]&(F_1(x)+F_1(y))+(F_2(x)+F_2(y))\ar[dl]\\
(F_1(x)+F_2(x))+(F_1(y)+F_2(y))&
}\]
The associativity constraints and the commutativity constraints for $\Hom(\calP_1,\calP_2)$ are clear.
If $\calP_1,\calP_2,\calP_3$ are three Picard groupoids, then $\Hom(\calP_1,\calP_2;\calP_3)$ is defined as $\Hom(\calP_1,\Hom(\calP_2,\calP_3))$, called the Picard groupoid of bilinear homomorphisms from $\calP_1\times\calP_2$ to $\calP_3$. The Picard groupoid of trilinear homomorphisms from $\calP_1\times\calP_2\times\calP_3$ to $\calP_4$ is defined similarly.

For a (small) monoidal group-like groupoid (or gr-category) $C$ we
denote by $\pi_0(C)$ the group\footnote{The group structure on
$\pi_0(C)$ is induced by the monoidal structure of $C$.} of
isomorphism classes of objects. We denote by $\pi_1(C)$ the group
$\Aut_C(e)$, where $e$ is the unit objet of $C$.  It follows that
$\pi_1(C)$ is an abelian group. If $C$ is a Picard groupoid, then
$\pi_0(C)$ is also an abelian group.

\subsection{Strictly commutative vs. non-strictly commutative Picard groupoids}\label{strict vs non-strict}
If the commutativity constraints $c$ further satisfy $c_{x,x}=\id$,
then the Picard groupoid $\calP$ is called strictly commutative. It
is a theorem of Deligne's (cf. \cite{Del}) that the $2$-category of
strictly commutative Picard groupoids is $2$-equivalent to the
$2$-category of $2$-term complexes of abelian groups concentrated on
degree $-1$ and $0$, whose degree $-1$ term abelian groups are
injective\footnote{In fact, Deligne's theorem holds in any topos.}.

\begin{ex}
The most famous example is $\calP=BA$, where $A$ is an abelian
group, and $BA$ is the category of $A$-torsors. The tensor products
of $A$-torsors make $BA$ a strictly commutative Picard groupoid. The
$2$-term complex of abelian groups that represents $BA$ under
Deligne's theorem is any injective resolution of $A[1]$. If
$A=k^\times$ is the group of invertible elements in a field $k$,
then $BA$ is also denoted by $\calP ic$, which is the symmetric
monoidal category of $1$-dimensional $k$-vector spaces.
\end{ex}

However, it is also important for us to consider the non strictly
commutative Picard groupoids. The following example of a
non-strictly commutative Picard groupoid is crucial.

\begin{ex}
Let $\calP ic^\bbZ$ denote the category of graded lines (i.e.
$1$-dimensional $k$-vector spaces with gradings), over a base field
$k$. An object in $\calP ic^\bbZ$ is a pair $(\ell,n)$ where $\ell$
is a $1$-dimensional $k$-vector space, and $n$ is an integer. The
morphism set $\Hom_{\calP ic^\bbZ}((\ell_1,n_1),(\ell_2,n_2)$ is
empty unless $n_1=n_2$, and in this case, it is just
$\Hom_k(\ell_1,\ell_2) \setminus 0$. Observe that as a groupoid,
$\calP ic^\bbZ$ is not connected. In fact $\pi_0(\calP)\simeq\bbZ$.
The tensor product $\calP ic^\bbZ\times\calP ic^\bbZ\to\calP
ic^\bbZ$ is given as
\[(\ell_1,n_1)\otimes(\ell_2,n_2)\mapsto (\ell_1\otimes\ell_2,n_1+n_2) \mbox{.}\]
There is a natural associativity constraint that makes $\calP
ic^\bbZ$ a monoidal groupoid.
\begin{rmk}
For the Picard groupoids $\calP ic$ and $\calP ic^\bbZ$, we will
often use in this article the usual notation $"\otimes"$ for
monoidal structures in these categories, although for a general
Picard groupoid we denoted it as "+".
\end{rmk}
 We note that the commutativity constraint in category $\calP ic^\bbZ$
is the interesting one. Namely,
\[c_{\ell_1,\ell_2}:(\ell_1\otimes\ell_2,n_1+n_2)\simeq(\ell_2\otimes\ell_1,n_2+n_1),\ \
c_{\ell_1,\ell_2}(v\otimes w)=(-1)^{n_1n_2}w\otimes v.\]

Of course, there is another commutativity constraint on the category
of graded lines given by $c(v\otimes w)=w\otimes v$. Then as a
Picard groupoid with this naive commutativity constraints, it is
just the strictly commutative Picard groupoid $\calP ic\times\bbZ$.
There is a natural monoidal equivalence $\calP ic^\bbZ\simeq\calP
ic\times\bbZ$, but this equivalence is NOT symmetric monoidal (i.e.
a $1$-homomorphism of Picard groupoids). We denote by
\[F_{\calP ic}:\calP ic^\bbZ  \to \calP ic\]
the natural monoidal functor.

The importance of $\calP ic^\bbZ$  lies in the following
observation.  Let us make the following convention.

\medskip

\noindent\bf Convention. \rm For any category $\calC$ we denote by
$(\calC, \on{isom})$ a category with the same objects as in the
category $\calC$, and morphisms in the category $(\calC, \on{isom})$
are the isomorphisms in the category $\calC$.

\medskip

Now let $\Tate_0$ be the category of finite dimensional vector
spaces over a field $k$. The categories $\Tate_0$ and $(\Tate_0,
\on{isom})$ are symmetric monoidal categories under the direct sum.
The commutativity constraints in the categories $\Tate_0$ and
$(\Tate_0, \on{isom})$  are defined in the natural way. Namely, the
map $c_{V,W}: V\oplus W\to W\oplus V$ is given by
$c_{V,W}(v,w)=(w,v)$. Then there is a natural symmetric monoidal
functor
\begin{equation}\label{det}
\det: (\Tate_0, \on{isom}) \to\calP ic^\bbZ,
\end{equation}
which assigns to every $V$ its top exterior power and the grading
$\dim V$, i.e. the dimension of the vector space $V$ over the field
$k$. Observe, however, that the functor $F_{\calP
ic}\circ\det:(\Tate_0,\on{isom})\to\calP ic$ is NOT symmetric
monoidal.
\end{ex}

It is a folklore theorem that the category of Picard groupoid (not
necessarily strict commutative) is equivalent to the category of
spectra whose only non-vanishing homotopy groups are $\pi_0$ and
$\pi_1$\footnote{Indeed, consider the geometrization of the nerve of $\calP$.
Then the Picard structure of $\calP$ puts an $E_\infty$-structure on this space.}. For example,
 $\calP ic^\bbZ$ should correspond to the truncation
$\tau_{\leq 1}\calK$, where $\calK$ is the spectra of algebraic
$K$-theory of $k$.

\subsection{$\calP$-torsors} \label{Ptorsors}
Let $\calP$ be a Picard groupoid. Recall (see
also~\cite[Appendix~A6]{BBE} and~\cite[\S~5.1]{Dr}) that a
$\calP$-torsor $\calL$ is a module category over $\calP$, i.e.,
there is a bifunctor
\[+:\calP\times\calL\to\calL\] together with natural isomorphisms
\[a_{x,y,v}: (x+y)+ v\simeq x+ (y+ v), \ \ x,y\in\calP, v\in\calL,\]
satisfying
\begin{enumerate}
\item[(i)] the pentagon axiom, i.e. a diagram similar to \eqref{pentagon} holds;
\item[(ii)] for any $x\in\calP$, the functor from $\calL$ to $\calL$ given by $v\mapsto x+ v$ is an equivalence;
\item[(iii)] for any $v\in\calL$, the functor
from $\calP$ to $\calL$ given by $x\mapsto x+ v$ is an equivalence of categories.
\end{enumerate}
It is clear that we can verify the  condition~(ii) of this
definition  only for the unit object $e $ of $\calP$.

For any $v\in\calL$, there is a unique isomorphism $e+v\simeq v$ such that the following diagram is commutative
\[\xymatrix{(e+e)+v\ar[rr]\ar[dr]&&e+(e+v)\ar[dl]\\
                   &e+v&}.\]

If $\calL_1,\calL_2$ are $\calP$-torsors, then $\Hom_\calP(\calL_1,\calL_2)$ is the category defined as follows.
Objects are 1-isomorphisms, i.e. equivalences $F:\calL_1\to\calL_2$ together with
isomorphisms $\lambda:F(x+v)\simeq x+F(v)$ such that the following diagram is commutative
\begin{equation}\label{la}\xymatrix{F((x+y)+v)\ar[r]\ar[d]&(x+y)+F(v)\ar[d]\\
F(x+(y+v))\ar[r]&x+(y+F(v))
}\end{equation}
Morphisms are natural transformations $\theta:F_1\to F_2$ such that the following diagram is commutative
\[\xymatrix{F_1(x+v)\ar[r]\ar_\theta[d]&x+F_1(v)\ar^\theta[d]\\
F_2(x+v)\ar[r]&x+F_2(v)}\]

From discussions above it follows that  all $\calP$-torsors form a
$2$-category, denoted by $B\calP$. We will choose, once and for all,
for any $\calP$-torsors $\calL_1,\calL_2$ and
$F\in\Hom_{\calP}(\calL_1,\calL_2)$, a quasi-inverse $F^{-1}$ of $F$
together with an isomorphism $F^{-1}F\simeq\on{id}$.

Moreover, $B\calP$ is a category enriched over itself. That is, for
any $\calP$-torsors $\calL_1,\calL_2$ the category
$\Hom_\calP(\calL_1,\calL_2)$ is again a $\calP$-torsor, where an
action of $\calP$ on $\Hom_\calP(\calL_1,\calL_2)$ is defined as
follows: for any $z \in \calP$, $v \in \calL_1$, $F \in
\Hom_\calP(\calL_1,\calL_2)$ we put $z+F \in
\Hom_\calP(\calL_1,\calL_2)$ as $(z+F)(v):=z+F(v)$. Now the
isomorphism $\lambda$ for the equivalence $z +F$ is defined by means
of the braiding maps $c$ in $\calP$ (commutativity constraints from
section~\ref{picgr}). Then the diagram~\eqref{la} for the
equivalence $z+F$ follows from hexagon diagram~\eqref{hexagon}. It
is clear that this definition is extended to the definition of a
bifunctor \begin{equation}\label{bi}+ : \calP \times
\Hom_\calP(\calL_1,\calL_2) \to
\Hom_\calP(\calL_1,\calL_2)\end{equation} such that the axioms  of
$\calP$-torsor are satisfied (see the beginning of this section).

We note that to prove that the category $B\calP$ is enriched over
itself we used the commutativity constraints in $\calP$. The
commutativity constraints will be important also below to define the
sum of two $\calP$-torsors.

 The category $B\calP$ furthermore forms a Picard
$2$-groupoid. We will not make the definition of Picard
$2$-groupoids precise. (However, one refers to \cite{KV,Br} for
details). We will only describe the Picard structure on  $B\calP$ in
the way we need.

First, if $\calL_1,\calL_2$ are two $\calP$-torsors, then
$\calL_1+\calL_2$ is defined to be the category whose objects are
pairs $(v,w)$, where $v\in\calL_1$ and $w\in\calL_2$. The morphisms
from $(v,w)$ to $(v',w')$ are defined as the equivalence classes of
triples $(x,\varphi_1,\varphi_2)$, where $x\in\calP$, $\varphi_1\in
\Hom_{\calL_1}(v,x+v')$ and $\varphi_2\in\Hom_{\calL_2}(x+w,w')$,
and $(x,\varphi_1,\varphi_2)\sim(y,\phi_1,\phi_2)$ if there exists a
map $f:x\to y$ such that $\phi_1=f(\varphi_1)$ and
$\varphi_2=f(\phi_2)$. The identity in
$\Hom_{\calL_1+\calL_2}((v,w),(v,w))$ and the composition
\[\Hom_{\calL_1+\calL_2}((v,w),(v',w'))\times
\Hom_{\calL_1+\calL_2}((v',w'),(v'',w''))\to\Hom_{\calL_1+\calL_2}((v,w),(v'',w''))\]
are clear. (To define the composition we have to use the
commutativity constraints in $B \calP$.) So $\calL_1+\calL_2$ is a
category. Define the action of $\calP$ on $\calL_1+\calL_2$ as
$x+(v,w):=(x+v,w)$. The natural isomorphism $(x+y)+(v,w)\simeq
x+(y+(v,w))$ is the obvious one. It is easy to check that
$\calL_1+\calL_2$ is a $\calP$-torsor.

There is an obvious $1$-isomorphism of $\calP$-torsors
\[A:(\calL_1+\calL_2)+\calL_3\simeq\calL_1+(\calL_2+\calL_3) \mbox{,}\] which is the associativity constraint.
Namely, objects in
$(\calL_1+\calL_2)+\calL_3$ and in $\calL_1+(\calL_2+\calL_3)$ are
both canonically bijective to triples $(v_1,v_2,v_3)$ where
$v_i\in\calL_i$. Then $A$ is identity on objects. A morphism from
$(v_1,v_2,v_3)$ to $(w_1,w_2,w_3)$ in $(\calL_1+\calL_2)+\calL_3$ is
of the form $(x,(y,\varphi_1,\varphi_2),\varphi_3)$, where
$x,y\in\calP$, $\varphi_1:v_1\to y+(x+w_1)$, $\varphi_2:y+v_2\to
w_2$, $\varphi_3:x+v_3\to w_3$. Then $A$ maps
$(x,(y,\varphi_1,\varphi_2),\varphi_3)$ to
$(x+y,\varphi'_1,(x,\varphi'_2,\varphi'_3))$, where
$\varphi'_1:v_1\to (x+y)+w_1$ coming from
$v_1\stackrel{\varphi_1}{\to}y+(x+w_1)\simeq(y+x)+w_1\simeq(x+y)+w_1$,
$\varphi'_2:(x+y)+v_2\to x+w_2$ coming from $(x+y)+v_2\simeq
x+(y+v_2)\stackrel{x+\varphi_2}{\to}x+w_2$ and $\varphi'_3:x+v_3\to
w_3$ is the same as $\varphi_3$.

To complete the definition of $A$, we should specify for every
$x\in\calP, (v_1,v_2,v_3)\in (\calL_1+\calL_2)+\calL_3$, an
isomorphism $\lambda:A(x+(v_1,v_2,v_3))\simeq x+A(v_1,v_2,v_3)$ such
that the diagram \eqref{la} is commutative for $F =A$. It is clear
that $\lambda=\on{id}:(x+v_1,v_2,v_3)=(x+v_1,v_2,v_3)$ will suffice
for this purpose.

It is clear from definition of $A$ that we can similarly construct a
$1$-morphism $A^{-1}$ of $\calP$-torsors such that the following
equalities are satisfied:
$$
A^{-1}A =A A^{-1} = \on{id} \mbox{.}
$$

From above construction of  the associativity constraints
($1$-morphisms $A$ and $A^{-1}$) it follows that for any
$\calP$-torsors $\calL_1, \calL_2, \calL_3, \calL_4$ the following
diagram of $1$-morphisms (pentagon diagram) is commutative
\begin{equation}\label{pentagon2}\xymatrix{&(\calL_1+\calL_2)+(\calL_3+\calL_4)\ar[dl]\ar[dr]&\\
\calL_1+(\calL_2+(\calL_3+\calL_4))\ar[d]&&((\calL_1+\calL_2)+\calL_3)+\calL_4\ar[d]\\
\calL_1+((\calL_2+\calL_3)+\calL_4)\ar[rr]&&(\calL_1+(\calL_2+\calL_3))+\calL_4}\end{equation}
(To prove this diagram we note that this diagram is evident for
objects from category $(\calL_1+\calL_2)+(\calL_3+\calL_4)$. To
verify this diagram for morphisms from this category one needs to
make some non-complicated routine calculations. The analogous
reasonings are also applied to the diagram~\eqref{hexagon2} below.)

The following axioms  are satisfied in the category $B\calP$ and
describe the functoriality of the associativity constraints. Let
$\calL_1, \calL_2, \calL_3, \calL_1'$ be any $\calP$-torsors, and
$\calL_1 \to \calL_1'$ be any $1$-morphism of $\calP$-torsors, then
the following diagram of $1$-morphisms is commutative
\begin{equation} \label{axas1}
\xymatrix{
 (\calL_{1} + \calL_{2}) + \calL_{3} \ar[r]  \ar[d] & (\calL_{1}' + \calL_{2}) + \calL_{3} \ar[d] \\
 \calL_{1} + (\calL_{2} + \calL_{3}) \ar[r] & \calL_{1}' + (\calL_{2}+ \calL_{3}) \mbox{.}
 }
 \end{equation}
 Let $\calL_1, \calL_2, \calL_3, \calL_2'$ be any $\calP$-torsors, and $\calL_2 \to \calL_2'$ be any $1$-morphism of $\calP$-torsors, then the following diagram of $1$-morphisms is commutative
 \begin{equation} \label{axas2}
 \xymatrix{
 (\calL_{1} + \calL_{2}) + \calL_{3} \ar[r]  \ar[d] & (\calL_{1} + \calL_{2}' ) + \calL_{3} \ar[d] \\
 \calL_{1} + (\calL_{2} + \calL_{3}) \ar[r] & \calL_{1} + (\calL_{2}' +  \calL_{3}) \mbox{.}
 }
 \end{equation}
 Let $\calL_1, \calL_2, \calL_3, \calL_3'$ be any $\calP$-torsors, and $\calL_3 \to \calL_3'$ be any $1$-morphism of $\calP$-torsors, then the following diagram of $1$-morphisms is commutative
\begin{equation} \label{axas3}
 \xymatrix{
 (\calL_{1} + \calL_{2}) + \calL_{3} \ar[r]  \ar[d] & (\calL_{1} + \calL_{2}) + \calL_{3}'  \ar[d] \\
 \calL_{1} + (\calL_{2} + \calL_{3}) \ar[r] & \calL_{1} + (\calL_{2} + \calL_{3}') \mbox{.}
 }
 \end{equation}
(In  diagrams~\eqref{axas1}-\eqref{axas3} the vertical arrows are the associativity constraints.)

Next we define the commutativity constraints. Recall that we have chosen for each $x\in\calP$ its
inverse $(-x,\phi_x)$, and then obtained the isomorphism \eqref{adjun}. This gives an
obvious $1$-isomorphism
\[C:\calL_1+\calL_2\simeq\calL_2+\calL_1 \mbox{.}\]
 Namely, $C$ will map the object $(v_1,v_2)$ to $(v_2,v_1)$,
and $(x,\varphi_1,\varphi_2):(v_1,v_2)\to(w_1,w_2)$ to $(-x,\varphi'_1,\varphi'_2):(v_2,v_1)\to(w_2,w_1)$,
where $\varphi'_1:v_2\simeq e+v_2\simeq(-x+x)+v_2\simeq-x+(x+v_2)\stackrel{-x+\varphi_2}{\to}-x+w_2$ and
$\varphi'_2:-x+v_1\stackrel{-x+\varphi_1}{\to}-x+(x+w_1)\simeq(-x+x)+w_1\simeq e+w_1\simeq w_1$.

We also define for each $x\in\calP$, $(v_1,v_2)\in\calL_1+\calL_2$, the
isomorphism $\lambda:C(x+(v_1,v_2))=(v_2,x+v_1)\to x+C(v_1,v_2)=(x+v_2,v_1)$
as $\lambda=(-x,\varphi_1,\varphi_2)$, where $\varphi_1:v_2\simeq (-x+x)+v_2\simeq -x+(x+v_2)$
and $\varphi_2:-x+(x+v_1)\simeq(-x+x)+v_1\simeq v_1$.

In addition, by \eqref{inverse}, there is an equality of $1$-morphisms
 $C^2 = \on{id}$.

The commutativity constrains together with the associativity constrains satisfy the  hexagon diagram, i.e. for any $\calP$-torsors
$\calL_1, \calL_2, \calL_3$ the following diagram of
$1$-morphisms  is commutative
\begin{equation}\label{hexagon2}\xymatrix{&(\calL_1+\calL_2)+\calL_3\ar[dl]\ar[dr]& \\
(\calL_1+\calL_2)+\calL_3\ar[d]&&\calL_1+(\calL_2+\calL_3)\ar[d]\\
\calL_1+(\calL_2+\calL_3)\ar[dr]&&\calL_1+(\calL_2+\calL_3)\ar[dl]\\
&(\calL_1+\calL_2)+\calL_3&}
\end{equation}

The following axiom  is satisfied in the category $B\calP$ and
describes the functoriality of the commutativity constraints. Let
$\calL_1, \calL_2, \calL_1'$ be any $\calP$-torsors, and $\calL_1
\to \calL_1'$ be any $1$-morphism of $\calP$-torsors, then the
following diagram of $1$-morphisms is commutative
\begin{equation} \label{axcom}
\xymatrix{
 \calL_{1} + \calL_{2}  \ar[r]  \ar[d] & \calL_{1}' + \calL_{2}  \ar[d] \\
 \calL_{2} + \calL_{1}  \ar[r] & \calL_{2} + \calL_{1}' \mbox{,}
 }
 \end{equation}
where the vertical arrows are the commutativity constraints.

By regarding $\calP$ as a $\calP$-torsor, there is a canonical
$1$-isomorphism of $\calP$-torsors $\calP+\calL\to\calL,
(x,v)\mapsto x+v$ satisfying the associativity and commutativity
constraints. This means that $\calP$ is the unit in $B\calP$. For
each $\calL\in\ B\calP$, we have an object
\[-\calL:= \Hom_\calP(\calL,\calP),\]
together with a natural 1-isomorphism of $\calP$-torsors
$\varphi_\calL:\calL+(-\calL)\simeq\calP$. This object is called an
inverse of $\calL$.

For $\calL$ a $\calP$-torsor, $\Hom_\calP(\calL,\calL)$ is a natural
monoidal groupoid (by composition). The natural homomorphism
\begin{equation}\label{zz}\calZ:\calP\to\Hom_\calP(\calL,\calL)\end{equation} given by
$\calZ(z)=z+\id$\footnote{Recall that we constructed the bifunctor
$+:\calP\times\Hom_\calP(\calL,\calL)\to\Hom_\calP(\calL,\calL)$ in
\eqref{bi}.} is a $1$-isomorphism of monoidal groupoids. We will fix
once and for all its inverse, i.e., we choose an $1$-isomorphism of
monoidal groupoids
\begin{equation} \label{zzz}
\calZ^{-1}:\Hom_\calP(\calL,\calL) \to \calP
\end{equation}
 together with a
$2$-isomorphism $\calZ^{-1}\circ\calZ\simeq\on{id}$.

\begin{rmk}
We constructed some "semistrict" version of Picard $2$-groupoid, because diagrams~\eqref{pentagon2}-\eqref{axcom}   are true in $B\calP$ for $1$-morphisms without consideration of additional $2$-morphisms which involve higher coherence axioms for braided monoidal $2$-categories as in~\cite{KV} and~\cite{BN}. Besides,
 from the equality $C^2 = \on{id}$ we  obtain at one stroke  that our 2-category $B\calP$ is strongly braided, i.e, the diagram (8.4.6) in \cite{Br} pp. 149 holds.
Let us mention that in \emph{loc. cit.}, the commutativity constraint $C$ is denoted by $R$.
\end{rmk}

\subsection{The case $H^1(BG,\calP)$.}
\label{hompic}
 Let $\calP$ be a Picard groupoid, and $G$ be a group.
Then we define $H^1(BG,\calP)$ to be the Picard groupoid of homomorphisms from $G$ to $\calP$.
That is, the objects are monoidal functors from $G$ to $\calP$, where $G$ is regarded as a
discrete monoidal category (i.e. the  monoidal groupoid, where objects are elements of the group $G$,
 and morphisms in this groupoid are only the unit morphisms of objects), and morphisms between these monoidal
 functors are monoidal natural transformations. In concrete terms, $f\in H^1(BG,\calP)$ is a functor $f:G\to\calP$,
 together with isomorphisms
\[f(gg')\simeq f(g)+f(g')\]
which are compatible with the associativity constraints. The
monoidal structure on $H^1(BG,\calP)$ is given by
$(f+f')(g)=f(g)+f(g')$. The natural isomorphism $(f+f')(gg')\simeq
(f+f')(g)+(f+f')(g')$ is  the obvious one. The associativity
constraints and the commutativity constraints on $H^1(BG,\calP)$ are
clear. Let $(e,\varphi)$ be a unit of $\calP$, and $e$ is regarded
as a discrete Picard groupoid with one object. Then $f:G\to\calP$ is
called trivial if it is isomorphic to $G\to e\to \calP$.

\begin{ex}If $\calP=BA$, then $H^1(BG,BA)$ is equivalent to the category of central extensions of $G$ by $A$ as Picard groupoids.
\end{ex}

Let $Z_2\subset G\times G$ be the subset of commuting elements, so
that if $G$ itself is an abelian group, then $Z_2=G\times G$. In
general, fix $g\in G$, then $Z_2\cap (G\times g)\simeq Z_2 \cap
(g\times G)\simeq Z_G(g)$, the centralizer of $g$ in $G$.

\begin{lem-def}\label{comm}There is a well defined anti-symmetric bimultiplicative map
$\on{Comm}(f):Z_2\to\pi_1(\calP)=\End_\calP(e)$.
\end{lem-def}
\begin{proof} The definition of $\on{Comm}(f)$
is as follows. For $g_1,g_2\in Z_2$, we have
\[f(g_1g_2)\simeq f(g_1)+f(g_2)\simeq f(g_2)+f(g_1)\simeq f(g_2g_1)=f(g_1g_2),\]
where the first and the third isomorphisms come from the
constraints for the homomorphism $f$, and the second isomorphism comes from the commutativity
constraints of the Picard groupoid $\calP$. We thus obtain an
element
\[\on{Comm}(f)(g_1,g_2)\in\Aut_\calP(f(g_1g_2))\simeq\pi_1(\calP).\]
Since $\calP$ is Picard, i.e., the commutativity constraints satisfy
$c_{f(g_1),f(g_2)}=c^{-1}_{f(g_2),f(g_1)}$, the map $\on{Comm}$ is
anti-symmetric. One can check directly by diagram that
$\on{Comm}(f)$ is also bimulitplicative (see the analogous
diagram~\eqref{constrC2} below).

Here we will give another proof of bimultiplicativity whose higher
categorical analogue we will use  in the proof of
lemma-definition~\ref{C2}. We construct the following category
$H_f$, where objects of $H_f$ are all possible expressions
\[f(g_1)+ \cdots + f(g_k):=(\cdots(f(g_1)+f(g_2))+f(g_3))+\cdots)+f(g_k), \mbox{ where }g_i \in G,\]
and morphisms in $H_f$ are defined as following:
\begin{multline*}
\Hom_{H_f} (f(g_{i_1}) + \ldots + f(g_{i_k}) \, , \, f(g_{j_1}) + \ldots + f(g_{j_l}))= \\
\left\{\begin{array}{ll} \emptyset & \mbox{if} \quad g_{i_1} \ldots g_{i_k} \ne
g_{j_1} \ldots g_{j_l} \mbox{;}\\
\Hom_{\calP} (f(g_{i_1}) + \ldots + f(g_{i_k}) \, , \, f(g_{j_1}) + \ldots + f(g_{j_l}))
&\mbox{if} \quad g_{i_1} \ldots g_{i_k} =
g_{j_1} \ldots g_{j_l} \mbox{.}
\end{array}\right.
\end{multline*}

The category $H_f$ is a monoidal group-like groupoid (or
$gr$-category), where the monoidal structure on $H_f$ is given in an
obvious way by using the associativity constraints in the category
$\calP$. We have $\pi_0 (H_f) = G$,  and $H_f$ is equivalent to the
trivial $gr$-category. We consider $\pi_1(\calP)$-torsor $E$ over
 $Z_2$ which is the commutator of $H_f$ (see~\cite[\S3]{Br2}). The fibre of $E$ over $(g_1,g_2) \in Z_2 $ is the set
 \[E_{g_1,g_2} = \Hom_{H_f} (f(g_1) + f(g_2), f(g_2) + f(g_1)).\] The $\pi_1(\calP)$-torsor $E$ has a natural
 structure of a weak biextension of $Z_2$
 by $\pi_1(\calP)$ (see~\cite[prop. 3.1]{Br2}), i.e. there are partial composition laws on $E$ which are compatible
 (see also \eqref{comp}). Now the commutativity constraints $c_{f(g_1),f(g_2)}$ give a section of $E$ over
 $Z_2$ which  is compatible with partial composition laws on $E$, i.e. "bimultiplicative".
 (The compatibility of this section with the composition laws follows at once
 from the
 definition of the partial composition laws on $E$ and the hexagon diagram~\eqref{hexagon}.) The other section of $E$
 which is compatible with partial composition laws on $E$ is obtained as the composition of  following
  two morphisms from
  definition of $f$:
 $f(g_1) +f(g_2) \simeq f(g_1g_2) = f(g_2g_1) \simeq f(g_2) +f(g_1) $. (The compatibility of this section
 with composition laws follows  from diagrams (3.10) and (1.4) of~\cite{Br2}, because of the compatibility of
 our homomorphism $f$ with the associativity constraints.) Now the difference between the first section and the
 second section coincides with  $\Comm(f)$, which is, thus, a bimultiplicative function, because both
 sections are "bimultiplicative".
\end{proof}

\begin{rmk} \label{bext}
In~\cite[\S2]{Br2} the notion of a weak biextension was intoduced
only for $Z_2=B\times B$ where $B$ is an abelian group. Here, we
generalize this notion by allowing $B$ to be non-commutative and by
replacing $B \times B$ by $Z_2$. But all the axioms for partial
composition laws in \emph{loc. cit.} are still applicable in this
setting. The same remark applies when we talk about
"$(2,2)$-extension" in \S~\ref{sec2}.
\end{rmk}

\begin{rmk} \label{2.5}
It is clear that if $f\simeq f'$ in $H^1(BG,\calP)$, then
$\Comm(f)=\Comm(f')$.
\end{rmk}
\begin{rmk} \label{usual}
 When $\calP=BA$, this construction reduces to the usual construction of inverse to the
 commutator pairing maps for central extensions.
\end{rmk}

\begin{cor}\label{Add}One has
$\on{Comm}(f+f')=\on{Comm}(f)+\on{Comm}(f')$.
\end{cor}
\begin{proof}
It can be easily checked directly by diagrams. See, for example,
analogous formulas and diagrams~\eqref{baer}-\eqref{baer2} below.
\end{proof}
\begin{cor} \label{2.8}
Assume that $G$ is abelian so that $Z_2=G\times G$. Then
$\on{Comm}(f)$ is trivial if and only if the 1-homomorphism $f$ is a
1-homomorphism of Picard groupoids. In particular, if the
homomorphism $f$ is trivial, then $\on{Comm}(f)$ is trivial.
\end{cor}
\begin{proof} It follows from diagram~\eqref{dz}.
\end{proof}
The above two corollaries together can be rephrased as by saying
that if $G$ is abelian, then there is an exact sequence of Picard
groupoids
\[1\to\Hom(G,\calP)\to H^1(BG,\calP)\to \Hom(\wedge^2G,\pi_1(\calP)).\]

\subsection{The case $H^2(BG,\calP)$}   \label{sec2}
If $\calP'$ is a Picard $n$-groupoid, and $G$ is a group, one should
be able to define $H^1(BG,\calP')$ as the Picard $n$-groupoid of
homomorphisms from $G$ to $\calP'$. When $n=1$, this is what we
discussed in the previous subsection. The next step for
consideration is $n=2$. Again, instead of discussing general Picard
2-groupoids, we will focus on the case when $\calP'=B\calP$, where
$\calP$ is a Picard groupoid. Then one can interpret
$H^1(BG,B\calP)$ as the Picard groupoid\footnote{As we just
mentioned, it is in fact a Picard $2$-groupoid.} of central
extensions of the group $G$ by the Picard groupoid $\calP$. For this
reason, we also denote $H^1(BG,B\calP)$ by $H^2(BG,\calP)$.

In concrete terms, an object $\calL$ in $H^2(BG,\calP)$ is a rule to
assign to every $g\in G$ a $\calP$-torsor $\calL_g$, and to every
$g,g'$ an equivalence $\calL_{gg'}\simeq \calL_g+\calL_{g'}$ of
$\calP$-torsors, and to every $g,g',g''$ an isomorphism between two
equivalences
\begin{equation}\label{assoc}
\xymatrix@C=20pt@R=20pt{&\calL_{gg'g''}\ar[dr]\ar[dl]
\ar@{}[dd]^(.4){\,}="1" \ar@{}[dd]^(.8){\,}="3"
\ar@{=>}"1";"3"
&
\\
\calL_{gg'}+\calL_{g''}\ar[d]&&\calL_g+\calL_{g'g''}\ar[d]\\
(\calL_g+\calL_{g'})+\calL_{g''}\ar[rr]&&\calL_g+(\calL_{g'}+\calL_{g''})}\end{equation}
such that for every $g,g',g'',g'''$, the natural compatibility
condition  holds, which we  describe below.

\begin{rmk}
Our notation for the $2$-arrow in diagram~\eqref{assoc} is symbolic,
and is distinct from the traditional notation of $2$-arrows in a
$2$-category, because this $2$-arrow is between a pair of $1$-arrows
from $\calL_{gg'g''}$ to $\calL_g + (\calL'_g + \calL_g'')$ and
should be written horizontally from left to right rather than
vertically. This notation for the $2$-arrow will be important for us
in diagram~\eqref{ip14}.
\end{rmk}

We define an isomorphism between two central extensions of $G$ by
$\calP$. An isomorphism between two central extensions
$\calL,\calL'$ is a rule which assigns to any $g$ a $\calP$-torsor
$1$-isomorphism $\calL_g\simeq\calL'_g$, and to any $g,g'$ the
following
 $2$-isomorphism
\[\xymatrix{\calL_{gg'}\ar[r]\ar[d]&\calL_g+\calL_{g'}\ar[d]
\ar@{}[d]_(.2){\, }="1"
\\
\calL'_{gg'}\ar[r]
\ar@{}[r]^(.2){\,}="2"
&\calL'_g+\calL'_{g'}
\ar@{}"1";"2"^(.2){\,}="3"
      \ar@{}"1";"2"^(.8){\,}="4"
      \ar@{=>}"4";"3"}
\mbox{.}\]
In
addition, these assignments have to be compatible with
diagram~\eqref{assoc} in an obvious way.

Now we describe the compatibility condition which we need after
diagram~\eqref{assoc}. If we don't consider the associativity
constraints in category $B\calP$, then the $2$-arrows induced by the
one in~\eqref{assoc} should satisfy the compatibility condition
described by  the following cube:
\begin{equation}
\label{cube}
\xymatrix
@C=1pt@R=37pt
{
& \calL_{g} + \calL_{g'} + \calL_{g''g'''}
\ar[rr] && \calL_{g} + \calL_{g'} + \calL_{g''} + \calL_{g'''} \\
\calL_{g} + \calL_{g'g''g'''}
\ar[ur] \ar[rr]
 & \ar@{-->}[u]&
\calL_{g} + \calL_{g'g''} +  \calL_{g'''} \ar[ur] \\
&  \calL_{gg'} + \calL_{g''g'''} \ar@{--}[u] \ar@{--}[r] &\ar@{-->}[r]&
\calL_{gg'} + \calL_{g''} + \calL_{g'''} \ar[uu] \\
\calL_{gg'g''g'''} \ar[uu] \ar@{-->}[ur] \ar[rr] && \calL_{gg'g''} + \calL_{g'''} \ar[uu] \ar[ur]
}
\end{equation}
To obtain the correct compatibility diagram for $2$-morphisms, we have to replace in diagram~\eqref{cube}
 the arrow (an edge of cube)
 $$\xymatrix{\calL_{gg'} + \calL_{g''} + \calL_{g'''} \ar[r]  &\calL_{g} + \calL_{g'} + \calL_{g''} + \calL_{g'''}}$$ by the following commutative diagram of $1$-morphisms in the category~$B\calP$
 \begin{equation} \label{ee1}
 \xymatrix{
 (\calL_{gg'} + \calL_{g''}) + \calL_{g'''} \ar[r]  \ar[d] & ((\calL_{g} + \calL_{g'}) + \calL_{g''}) + \calL_{g'''} \ar[d] \\
 \calL_{gg'} + (\calL_{g''} + \calL_{g'''}) \ar[r] & (\calL_{g} + \calL_{g'}) + (\calL_{g''}+ \calL_{g'''})
 }
 \end{equation}
 (where the vertical arrows are  associativity constraints); we have to replace in diagram~\eqref{cube} the arrow
 (an edge of the cube)
 $$\xymatrix{\calL_{g} + \calL_{g'g''} + \calL_{g'''} \ar[r]  &\calL_{g} + \calL_{g'} + \calL_{g''} + \calL_{g'''}}$$
 by the following commutative diagram of $1$-morphisms in the category~$B\calP$
 \begin{equation} \label{ee2}
 \xymatrix{
 (\calL_{g} + \calL_{g'g''}) + \calL_{g'''} \ar[r]  \ar[d] & (\calL_{g} + (\calL_{g'} + \calL_{g''})) + \calL_{g'''} \ar[d] \\
 \calL_{g} + (\calL_{g'g''} + \calL_{g'''}) \ar[r] & \calL_{g} + ((\calL_{g'} + \calL_{g''})+ \calL_{g'''})
 }
 \end{equation}
 (where vertical arrows are  associativity constraints); we have to replace in diagram~\eqref{cube} the arrow
 (an edge of the cube)
 $$\xymatrix{\calL_{g} + \calL_{g'} + \calL_{g''g'''} \ar[r]  &\calL_{g} + \calL_{g'} + \calL_{g''} + \calL_{g'''}}$$
 by the following commutative diagram of $1$-morphisms in the category~$B\calP$
 \begin{equation} \label{ee3}
 \xymatrix{
 (\calL_{g} + \calL_{g'}) + \calL_{g''g'''} \ar[r]  \ar[d] & (\calL_{g} + \calL_{g'}) + (\calL_{g''} + \calL_{g'''}) \ar[d] \\
 \calL_{g} + (\calL_{g'} + \calL_{g''g'''}) \ar[r] & \calL_{g} + (\calL_{g'} + (\calL_{g''}+ \calL_{g'''}))
 }
 \end{equation}
 (where vertical arrows are  associativity constraints). Besides, instead of
 the vertex $\calL_{g} + \calL_{g'} + \calL_{g''} + \calL_{g'''}$
 in diagram~\eqref{cube}
 we insert the commutative diagram which is  the modification of pentagon diagram~\eqref{pentagon2} for
 $\calL_{g}, \calL_{g'}, \calL_{g''}, \calL_{g'''}$, and this diagram  is always true in category $B\calP$.
 The  correct compatibility diagram for $2$-morphisms
from diagrams~\eqref{assoc} has  $15$ vertices.

We note that diagrams~\eqref{ee1}-\eqref{ee3} are commutative for
$1$-morphisms, i.e. the corresponding $2$-isomorphisms  equal
identity morphisms. These diagrams  express the "functoriality" of
associativity constraints in $B\calP$ and follow from
axioms-di\-a\-grams~\eqref{axas1}-\eqref{axas3}   in category
$B\calP$.

\medskip

The trivial central extension of $G$ by $\calP$ is the rule which
assign to every $g\in G$ the trivial $\calP$-torsor $\calP$, to
every $g,g'$ the natural $1$-isomorphism $\calP\simeq
\calP+\calP$\footnote{The naturality means that this $1$-isomorphism
is the chosen quasi-inverse of the natural $1$-isomorphism
$\calP+\calP\to\calP$.}, and to every $g,g',g''$ the corresponding
natural $2$-isomorphism.

\begin{rmk}A central extension $\calL$ of $G$ by $\calP$ gives rise
to a gr-category, $\tilde{\calL}$, together with a short exact
sequence of gr-categories in the sense of \cite[definition
2.1.2]{Br3}
\[1\to\calP\stackrel{i}{\to}\tilde{\calL}\stackrel{\pi}{\to} G\to 1.\]
Namely, as a category, $\tilde{\calL}=\bigcup_{g\in G} \calL_g$.
Then the natural equivalence $\calL_{gg'}\simeq \calL_g+\calL_{g'}$
together with the compatibility conditions endows $\tilde{\calL}$
with a gr-category structure. The natural morphism
$\pi:\tilde{\calL}\to G$ is clearly monoidal, and one can show that
$\ker\pi=\calL_e$ is 1-isomorphic to $\calP$.

As is shown in \emph{loc. cit.}, such a short exact sequence endows
every $\tilde{\calL}_g:=\pi^{-1}(g)=\calL_g$ with a $\calP$-bitorsor
structure. This $\calP$-bitorsor structure is nothing but the
canonical $\calP$-bitorsor structure on $\calL_g$ (observe that the
morphism $\calZ:\calP\to\Hom_\calP(\calL_g,\calL_g)$ as in
\eqref{zz} induces a canonical $\calP$-bitorsor structure on
$\calL_g$).

The upshot is that an object $\calL$ in $H^2(BG,\calP)$ gives rise
to a categorical generalization of a central extension of a group by
an abelian group. This justifies our terminology. Indeed, one can
define a central extension of $G$ by $\calP$ as a short exact
sequence as above such that the induced $\calP$-bitorsor structure
on each $\tilde{\calL}_g$ is the canonical one induced from its left
$\calP$-torsor structure. Since we do not use this second
definition, we will not make it precise.
\end{rmk}

Finally, let us define the Picard structure on $H^2(BG,\calP)$. Let
$\calL$ and $\calL'$ be two central extensions of $G$ by $\calP$.
Then we define the central extension  $\calL + \calL'$ by the
following way:
$$
(\calL + \calL')_g := \calL_g + \calL'_g  \mbox{,}
$$
and the equivalence $(\calL + \calL')_{gg'}  \simeq (\calL+ \calL')_g + (\calL + \calL')_{g'} $ as the composition of the following equivalences
\begin{multline*}
(\calL + \calL')_{gg'} = \calL_{gg'} + \calL'_{gg'} \simeq (\calL_g + \calL_{g'})+ (\calL'_{g} + \calL'_{g'}) \\
\simeq (\calL_g +\calL'_{g}) + (\calL_{g'}+ \calL'_{g'})=(\calL + \calL' )_g + (\calL + \calL')_{g'}  \mbox{.}
\end{multline*}
The corresponding $2$-isomorphism for central extension $\calL +
\calL'$ and any elements $g, g', g''$ of $G$ follows from
diagrams~\eqref{assoc} for central extensions $\calL$ and $\calL'$.
The further compatibility conditions for these $2$-isomorphisms hold
as in diagrams~\eqref{cube}-\eqref{ee3}, since they follow at once
from the corresponding diagrams for central extensions $\calL$  and
$\calL'$.

\medskip

Again, let $Z_2$ denote the subset of $G\times G$ consisting of
commuting elements. We will give a categorical analogue of
lemma-definition~\ref{comm}. For this purpose, let us first explain
some terminology. A 1-morphism $f:Z_2\to\calP$ is called
bimultiplicative if for fixed $g\in G$, $(Z_G(g),g)\subset
Z_2\to\calP$ and $(g,Z_G(g))\subset Z_2\to\calP$ are homomorphisms,
i.e. monoidal functors from discrete monoidal categories
$(Z_G(g),g)$ and $(g,Z_G(g))$ to $\calP$. In addition, the following
diagram must be commutative (which is the compatibility condition
between these two homomorphisms)
\begin{equation}\label{comp}\small
\xymatrix{f(g_1g_2,g_3)+f(g_1g_2,g_4)\ar[r]&(f(g_1,g_3)+f(g_2,g_3))+(f(g_1,g_4)+f(g_2,g_4))\ar^\simeq[dd]\\
f(g_1g_2,g_3g_4)\ar^\simeq[u]\ar_\simeq[d]&\\
f(g_1,g_3g_4)+f(g_2,g_3g_4)\ar[r]&(f(g_1,g_3)+f(g_1,g_4))+(f(g_2,g_3)+f(g_2,g_4))}.\end{equation}
When $\calP=BA$, a bimultiplicative 1-morphism from $Z_2\to BA$ is the same as a weak biextension of $Z_2$ by $A$ as defined in \cite[\S 2]{Br2} (see also remark~\ref{bext}).

A 1-morphism $f:Z_2\to\calP$ is called anti-symmetric if there is a
2-isomorphism $\theta:f\simeq -f\circ\sigma$, where $\sigma$ is the
natural flip on $Z_2$, such that for any $(g_1,g_2)\in Z_2$, the
following diagram is commutative
\[\xymatrix{f(g_1,g_2)\ar^\simeq[r]\ar@{=}[d]&-f(g_2,g_1)\ar^\simeq[d]\\
f(g_1,g_2)&-(-f(g_1,g_2))\ar_\simeq[l]}\]

We need some more terminology. Following \cite[\S 7]{Br2}, we {\em
define}  a weak $(2,2)$-extension of $Z_2$ by $\calP$ as a rule
which assigns to every $(g,g')\in Z_2$ a $\calP$-torsor
$\calE_{(g,g')}$ such that its restrictions to $(g,Z_G(g))$ and
$Z_G(g),g)$ are central extensions of $Z_G(g)$ by $\calP$, and that
the corresponding diagram \eqref{comp} is $2$-commutative (i.e.
commutative modulo some $2$-isomorphism), and these $2$-isomorphisms
satisfy further compatibility conditions (see (7.1), (7.3) in {\it
loc. cit.} where these compatibility conditions are carefully spelt
out).

\begin{lem-def}\label{C2}There is an anti-symmetric bimultiplicative homomorphism
$C^\calL_2:Z_2\to\calP$.
\end{lem-def}
\begin{proof}As in the proof of lemma-definition~\ref{comm}, using the commutativity constraints
$C:\calL_{g}+\calL_{g'}\simeq \calL_{g'}+\calL_g$ in the category $B\calP$, one constructs the following composition
of $1$-isomorphisms:
\[\calL_{gg'}\simeq\calL_g+\calL_{g'}\simeq\calL_{g'}+\calL_g\simeq\calL_{g'g}=\calL_{gg'},\]
for $(g,g')\in Z_2$. In this way, we obtain a functor
$Z_2\to\Hom_\calP(\calL_{gg'},\calL_{gg'})$. Using
$\calZ^{-1}:\Hom_{\calP}(\calL_{gg'},\calL_{gg'})\to\calP$
(see~\eqref{zzz}), we get a morphism $C^\calL_2:Z_2\to\calP$.

We need to construct the following canonical isomorphisms
\[C^\calL_2(gg',g'')\simeq C^\calL_2(g,g'')+C^\calL_2(g',g''),
\quad C^\calL_2(g,g'g'')\simeq C^\calL_2(g,g')+C^\calL_2(g,g''),\]
satisfying the natural compatibility conditions. We now construct
the first isomorphism. The second is similar. Let
$\calZ:\calP\to\Hom_{\calP}(\calL_{gg'g''},\calL_{gg'g''})$ be the
canonical equivalence as monoidal groupoids as in \eqref{zz}. It is
enough to construct a canonical $2$-isomorphism
$\calZ(C^\calL_2(gg',g''))\simeq\calZ(C^\calL_2(g,g'')+C^\calL_2(g',g''))$.

By the definition of the morphism  $C^\calL_2$, there is a canonical
$2$-isomorphism from $1$-isomorphism
$\calZ(C_2^{\calL}(g,g'')+C_2^{\calL}(g',g''))$ to the following
composition of $1$-isomorphisms:
\begin{multline}\label{C2I}\calL_{gg'g''}\simeq\calL_{g}+\calL_{g'g''}\simeq\calL_g+(\calL_{g'}+\calL_{g''})\simeq\calL_g+(\calL_{g''}+\calL_{g'})\simeq\calL_g+\calL_{g''g'}\simeq\calL_{gg''g'}\\
\simeq\calL_{gg''}+\calL_{g'}\simeq(\calL_g+\calL_{g''})+\calL_{g'}\simeq(\calL_{g''}+\calL_g)+\calL_{g'}\simeq\calL_{gg''}+\calL_{g'}\simeq\calL_{gg'g''}\end{multline}
By the definition of the central extension of $G$ by $\calP$ (see diagram~\eqref{assoc}), there is a canonical $2$-isomorphism from the above composition of $1$-isomorphisms to the following composition of $1$-isomorphisms
\begin{multline}\label{C2II}\calL_{gg'g''}\simeq\calL_{g}+\calL_{g'g''}\simeq\calL_g+(\calL_{g'}+\calL_{g''})\simeq\calL_g+(\calL_{g''}+\calL_{g'})\\
\simeq(\calL_g+\calL_{g''})+\calL_{g'}\simeq(\calL_{g''}+\calL_g)+\calL_{g'}\simeq\calL_{gg''}+\calL_{g'}\simeq\calL_{gg'g''}\end{multline}
From the hexagon axiom for $1$-morphisms in the category $B\calP$
(see diagram~\ref{hexagon2}) we have that  the  $1$-isomorphism
which is the composition of the above $1$-isomorphisms is equal to
the $1$-isomorphism which is the composition of the following
$1$-isomorphisms
\begin{multline}\label{C2III}\calL_{gg'g''}\simeq\calL_{g}+\calL_{g'g''}\simeq\calL_g+(\calL_{g'}+\calL_{g''})\simeq(\calL_g+\calL_{g'})+\calL_{g''}\\
\simeq\calL_{g''}+(\calL_g+\calL_{g'})\simeq(\calL_{g''}+\calL_g)+\calL_{g'}\simeq\calL_{gg''}+\calL_{g'}\simeq\calL_{gg'g''}\end{multline}
By the "functoriality" of the commutativity constraints in the category $B\calP$ (see axiom-diagram~\eqref{axcom})
 we have that  the  $1$-isomorphism which is the composition of the above $1$-isomorphisms is equal to the
 $1$-isomorphism which is the composition of the following $1$-isomorphisms
\begin{multline}\label{C2IV}\calL_{gg'g''}\simeq\calL_{g}+\calL_{g'g''}\simeq\calL_g+(\calL_{g'}+\calL_{g''})\simeq(\calL_g+\calL_{g'})+\calL_{g''}
\simeq \calL_{gg'} + \calL_{g''}\\ \simeq \calL_{g''} + \calL_{gg'}
\simeq\calL_{g''}+(\calL_g+\calL_{g'})\simeq(\calL_{g''}+\calL_g)+\calL_{g'}\simeq\calL_{gg''}+\calL_{g'}\simeq\calL_{gg'g''}\end{multline}
Again, by the definition of the central extension of $G$ by $\calP$ (see diagram~\eqref{assoc}, which we  apply twice now), there is a canonical $2$-isomorphism from the above composition of $1$-isomorphisms to the following composition of $1$-isomorphisms
\begin{equation}\label{C2V}\calL_{gg'g''}\simeq\calL_{gg'}+\calL_{g''}\simeq\calL_{g''}+\calL_{gg'}\simeq\calL_{gg'g''} \mbox{,}\end{equation}
which is canonically isomorphic to $\calZ(C_2^\calL(gg',g''))$.

Let us write down a diagram which will represent the above
$2$-isomorphisms. To simplify the notation, we will denote the
$2$-commutative diagram \eqref{assoc} as
\begin{equation} \label{ip14}
\xymatrix{\calL_{gg'g''}\ar@{=>}[d]\\\calL_g+\calL_{g'}+\calL_{g''}}
\end{equation}

Then, the $2$-isomorphism $\calZ(C^\calL_2(gg',g''))\simeq\calZ(C^\calL_2(g,g''))+\calZ(C^\calL_2(g',g''))$ is
represented by the following diagram
\begin{equation}\label{constrC2}
\xymatrix@C=20pt@R=20pt{& & \calL_{gg''g'} \ar@{=>}[d] \ar^{\calZ(C_2^\calL(g,g''))}[ddrr] & & \\
& & \calL_g+\calL_{g''}+\calL_{g'} \ar[dr] \ar[dl] & & \\
\calL_{gg'g''} \ar^{\calZ(C_2^\calL(g',g''))}[uurr] \ar@{=>}[r] \ar@ / _1.5pc/[rrrr]_{\calZ(C_2^\calL(gg',g''))} & \calL_g+\calL_{g'}+\calL_{g''}\ar[rr] && \calL_{g''}+\calL_g+\calL_{g'} & \calL_{g''gg'}\ar@{=>}[l]
}\end{equation}

To check all the compatibility conditions between these canonical
isomorphisms we generalize the proof of lemma-definition~\ref{comm}.
We construct the following $2$-category $H_{\calL}$, where objects
of $H_{\calL}$ are  objects from  categories given by all
expressions
\[\calL_{g_1}+ \cdots + \calL_{g_k}:=(\cdots(\calL_{g_1}+\calL_{g_2})+\calL_{g_3})+\cdots)+\calL_{g_k},
\mbox{ where }g_i \in G,\]
$1$-morphisms in $H_{\calL}$ are defined as following:
\begin{multline*}
\Hom_{H_{\calL}} (\calL_{g_{i_1}} + \ldots + \calL_{g_{i_k}} \, , \, \calL_{g_{j_1}} + \ldots + \calL_{g_{j_l}}) = \\
\left\{\begin{array}{ll}\emptyset & \mbox{if} \quad g_{i_1} \ldots g_{i_k} \ne
g_{j_1} \ldots g_{j_l} \mbox{;}\\
\Hom_{B\calP} (\calL_{g_{i_1}} + \ldots + \calL_{g_{i_k}} \, , \, \calL_{g_{j_1}} + \ldots + \calL_{g_{j_l}})
& \mbox{if} \quad g_{i_1} \ldots g_{i_k} =
g_{j_1} \ldots g_{j_l} \mbox{,}
\end{array}\right.
\end{multline*}
and $2$-morphisms in $2$-category $H_{\calL}$ come from
$2$-morphisms of category $B\calP$. The category $H_{\calL}$ is a
monoidal group-like $2$-groupoid (or a $2$-$gr$-category),
see~\cite[\S 8]{Br}, where monoidal structure on $H_{\calL}$ is
given in an obvious way by using the associativity constraints in
the category $B\calP$ and pentagon diagram~\eqref{pentagon2}. We
have $\pi_0 (H_{\calL}) = G$. We consider the $\calP$-torsor
$\calE^\calL$ on
 $Z_2$ which is the commutator of $H_{\calL}$ (see~\cite[\S8]{Br2}\footnote{L. Breen assumed
 for simplicity in {\it loc. cit.}
  that the group $\pi_1$ of a $2$-gr-category is equal to $0$. We have $\pi_1(H_{\calL}) \ne 0$, but
  the constructions and its properties which  we need remain true in our situation.}).
  The fibre of $\calE^\calL$ over $(g_1,g_2) \in Z_2 $ is the $\calP$-torsor
 \begin{equation*}\label{biex}
 \calE^\calL_{g_1,g_2} = \Hom_{H_{\calL}} (\calL_{g_1} + \calL_{g_2}, \calL_{g_2} + \calL_{g_1}).
 \end{equation*}
 The $\calP$-torsor $\calE^\calL$ on $Z_2 $ has a natural structure of a weak $(2,2)$-extension
 (see~\cite[prop. 8.1]{Br2}), i.e. there are partial composition (group) laws on $\calE^\calL$
 which are compatible (see diagrams (7.1), (7.3)  in {\it loc.cit}).
 Now the commutativity constraints $C$ from $B\calP$ give a trivialization of $\calP$-torsor $\calE^\calL$ on $Z_2$
 which is compatible with partial composition laws on $\calE^\calL$, i.e. "bimultiplicative".
 (The compatibility of this trivialization with composition laws follows at once from
  definition of partial composition laws on $\calE^\calL$ and hexagon diagram~\eqref{hexagon2}.
  See also the discussion in the end  of~\cite[\S8]{Br2} regarding the braiding structure in $H_{\calL}$,
  which gives the "bimultiplicative" trivialization of the $\calP$-torsor $\calE^\calL$ on $Z_2$.) The other
  trivialization of the $\calP$-torsor $\calE^\calL$ on $Z_2$ which is compatible with partial composition laws on
  $\calE^\calL$ is obtained as the composition of the following
  two equivalences from
   definition of $\calL$:
 \begin{equation*}\label{S}S_{\calL_{g_{1}},\calL_{g_{2}}}:\calL_{g_{1}} +\calL_{g_{2}}
 \simeq \calL_{g_{1}g_{2}} = \calL_{g_{2}g_{1}} \simeq \calL_{g_{2}} +\calL_{g_{1}}.\end{equation*}
Now the difference between the first  trivialization and the second
trivialization of the $\calP$-torsor $\calE^\calL$ on $Z_2$
coincides with  $C^\calL_2$, which is, thus, a bimultiplicative
homomorphism, because both trivializations are "bimultiplicative".

We have shown that $C_2^\calL:Z_2\to\calP$ is a bimultiplicative
1-morphism. One readily checks from the above constructions that
this is anti-symmetric from $Z_2$ to $\calP$, since $C^2 = \on{id}$.
\end{proof}

\begin{rmk} \label{Del}
If $\calP =BA$, then the construction of~$C^\calL_2$ given above is
equivalent to the construction of  the  commutator category of the
central extension $-\calL$ introduced by Deligne in~\cite{Del2}.
\end{rmk}

We also have the following categorical analogue of
corollary~\ref{Add}. First, let us remark that if
$f_1,f_2:Z_2\to\calP$ are two bimultiplicative homomorphisms, one
can define  $f_1+f_2$, which is again a bimultiplicative
homomorphism, in the same way as defining the Picard structure on
$H^1(BG,\calP)$.

\begin{lem} \label{lembaer} For any two central extensions $\calL$ and $\calL'$ of $G$ by $\calP$
there is a natural bimultiplicative $2$-isomorphism (i.e., it respects
the bimultiplicative structure) between bimultiplicative $1$-morphisms
$C^{\calL + \calL' }_2$ and $C^{\calL}_2 + C^{\calL'}_2$.
\end{lem}
\begin{proof} Recall that we have the following canonical $1$-isomorphism $$\calZ:\calP\to\Hom(\calL_{gg'}+\calL'_{gg'},\calL_{gg'}+\calL'_{gg'}) \mbox{.}$$
We construct a canonical isomorphism  $$\calZ(C^{\calL + \calL'}_2(g,g')) \simeq \calZ(C^{\calL}_2(g,g') + C^{\calL'}_2(g,g')) $$  for any $(g,g') \in Z_2$ as following.
By definition,  $\calZ(C^{\calL + \calL'}_2(g,g'))$ is canonically $2$-isomorphic to the following composition of $1$-morphisms
\begin{multline} \label{baer}
(\calL +\calL')_{gg'} = \calL_{gg'} + \calL_{gg'}' \simeq (\calL_{g}+\calL_{g'})+(\calL_g'+\calL'_{g'}) \simeq
(\calL_g + \calL_g') +(\calL_{g'}+ \calL_{g'}')  \\ =(\calL + \calL')_{g} + (\calL + \calL')_{g'}
\simeq (\calL + \calL')_{g'}  + (\calL + \calL')_{g} \simeq (\calL_{g'}+\calL_{g'}') +(\calL_g + \calL'_g) \\  \simeq
(\calL_{g'}+\calL_g) + (\calL_{g'}'  + \calL'_g )  \simeq \calL_{g'g} + \calL'_{g'g} = (\calL + \calL')_{g'g}
= (\calL + \calL')_{gg'}
\end{multline}
Using the "functoriality" of commutativity constraints, i.e.
applying twice  diagram~\eqref{axcom}, and using the following
commutative diagram (which is written without associativity
constraints)
 \begin{equation} \label{baer'}\xymatrix{&\calL_g+\calL_{g'}+\calL'_g+\calL'_{g'}\ar[dl]\ar[dr]&\\
\calL_g+\calL_{g'}+\calL'_{g'}+\calL'_g \ar[d] \ar[drr] && \calL_g+\calL'_g+\calL_{g'}+\calL'_{g'} \ar[d] \ar[ll]\\
\calL_{g'}+\calL_g+\calL'_{g'}+\calL'_g &&  \ar[ll] \calL_{g'}+\calL'_{g'}+\calL_g+\calL'_g} \end{equation}
 (to obtain the correct diagram we have to replace every triangle in this diagram by a hexagon coming from~\eqref{hexagon2}), we obtain that
the composition~\eqref{baer} of $1$-morphisms is equal to the following composition of $1$-morphisms
\begin{multline} \label{baer2}
 \calL_{gg'} + \calL_{gg'}' \simeq (\calL_{g}+\calL_{g'})+(\calL_g'+\calL'_{g'})  \\ \simeq
(\calL_{g'}+\calL_g) + (\calL_{g'}'  + \calL'_g )  \simeq \calL_{g'g} + \calL'_{g'g} = \calL_{gg'} + \calL'_{gg'} \mbox{,}
\end{multline}
which is, by definition,  $2$-isomorphic to  $\calZ(C^{\calL}_2(g,g') + C^{\calL'}_2(g,g'))$.

To complete the proof, we need to show that the following diagram
\begin{equation}\label{ccoomm} \small \xymatrix{C_2^{\calL+\calL'}(g,g'')+C_2^{\calL+\calL'}(g',g'')\ar[r]&
(C_2^{\calL}(g,g'')+C_2^{\calL'}(g,g''))+(C_2^{\calL}(g',g'')+C_2^{\calL'}(g',g''))\ar[dd]\\
C_2^{\calL+\calL'}(gg',g'')\ar[u]\ar[d]&\\
C_2^\calL(gg',g'')+C_2^{\calL'}(gg',g'')\ar[r]&(C_2^{\calL}(g,g'')+C_2^{\calL}(g',g''))+(C_2^{\calL'}(g,g'')+C_2^{\calL'}(g',g''))
}\end{equation} and a similar diagram involving
$C_2^{\calL+\calL'}(g,g'g'')$ are commutative. To prove this, let us
recall that the $2$-isomorphism $C_2^\calL(gg',g'')\simeq
C_2^\calL(g,g'')+C_2^\calL(g',g'')$ is the composition of the
following $2$-isomorphisms
\[\calZ(C_2^\calL(g,g'')+C_2^\calL(g',g''))\to\eqref{C2I}\to\eqref{C2II}\to\cdots\to\eqref{C2V}\to\calZ(C_2^\calL(gg',g'')) \mbox{.}\]
Let us denote the $1$-isomorphism \eqref{C2I} for $\calL$ (resp.
$\calL'$, resp. $\calL+\calL'$) as $\eqref{C2I}_\calL$ (resp.
$\eqref{C2I}_{\calL'}$, resp. $\eqref{C2I}_{\calL+\calL'}$) and etc.
Then it is readily  checked that there exists a canonical
$2$-isomorphism
\[\eqref{C2I}_{\calL}+\eqref{C2I}_{\calL'}\simeq\eqref{C2I}_{\calL+\calL'}\]
 between corresponding $1$-isomorphisms $\calL_{gg'g''}+\calL'_{gg'g''}\to\calL_{gg'g''}+\calL_{gg'g''}$, and
   canonical $2$-isomorphisms for \eqref{C2II}-\eqref{C2V} such that the following diagram
\[\xymatrix{\eqref{C2I}_\calL+\eqref{C2I}_{\calL'}\ar[r]\ar[d]&\eqref{C2I}_{\calL+\calL'}\ar[d]\\
\eqref{C2II}_\calL+\eqref{C2II}_{\calL'}\ar[r]&\eqref{C2II}_{\calL+\calL'}
}\]
and similar diagrams for \eqref{C2II}-\eqref{C2V} commute. In addition, the following diagrams commute.
\[\hspace{-0.35cm}\small\xymatrix{\calZ(C_2^{\calL+\calL'}(g,g'')+C_2^{\calL+\calL'}(g',g'')) \ar[r]\ar[d]&
\calZ((C_2^{\calL}(g,g'')+C_2^{\calL}(g',g''))+(C_2^{\calL'}(g,g'')+C_2^{\calL'}(g',g'')))\ar[d]\\
\eqref{C2I}_{\calL+\calL'}\ar[r]&\eqref{C2I}_{\calL}+\eqref{C2I}_{\calL'}
}\]
\[\xymatrix{\calZ(C_2^{\calL+\calL'}(gg',g''))\ar[r]&\calZ(C_2^{\calL}(gg',g'')+C_2^{\calL'}(gg',g''))\\
\eqref{C2V}_{\calL+\calL'}\ar[r] \ar[u]&\eqref{C2V}_{\calL}+\eqref{C2V}_{\calL'} \ar[u]
}\]
These facts together imply the commutativity of diagram \eqref{ccoomm}.
\end{proof}

Fix $g\in G$, the induced map $Z_G(g)\to\calP$ given by $g'\mapsto
C^\calL_2(g,g')$ is denoted by $C^\calL_g$. The bimultiplicativity
of $C_2^\calL$ implies that $C_g^\calL$ is an object in
$H^1(BZ_G(g),\calP)$. It is easy to see from the definition the
following lemma:

\begin{lem}(i) If two central extensions $\calL$ and $\calL'$ of $G$ by $\calP$ are isomorphic in $H^2(BG,\calP)$,
then for any $g$ the induced two homomorphisms $C^\calL_g$ and
$C^{\calL'}_g$ are isomorphic in $H^1(BZ_G(g),\calP)$.

(ii) One has that $C_g^\calP$ is the trivial homomorphism for any
$g\in G$.
\end{lem}

Let $Z_3\subset G\times G\times G$ be the subset of pairwise
commuting elements.

\begin{prop}\label{C3}Let a map
\[C^\calL_3:Z_3\to\pi_1(\calP)\]
be defined by
\[C^\calL_3(g,g',g''):=\on{Comm}(C^\calL_{g})(g',g'').\]
Then $C^\calL_3$ is an anti-symmetric tri-multiplicative
homomorphism from $Z_3$ to $\pi_1(\calP)$.
\end{prop}
\begin{proof} Let us see the
tri-multiplicativity of the map $C^\calL_3$. The multiplicativity of
this map with respect to $g'$ or $g''$ follows from
lemma-definition~\ref{comm}. The multiplicativity of this map with
respect to $g$ follows from lemma-definition~\ref{C2}  and
corollary~\ref{Add}.

The hard part is to prove now that the map $C^\calL_3$ is
anti-symmetric. Let us write $C_2$ instead of $C_2^\calL$, and $C_3$
instead of $C_3^{\calL}$  for simplicity. Let $(g,g',g'')\in Z_3$.
First of all, let us observe that by definition, there is a
canonical isomorphism
\begin{equation}\label{comm=comm}C_2(g,g'g'')+C_2(g',g'')\simeq C_2(g',g'')+C_2(g,g''g')\end{equation} induced by the following 2-commutative diagram
\[\xymatrix{
\calL_{gg'g''}\ar[r]\ar@/^1.5pc/|{\calZ(C_2(g,g'g''))}[rrr]\ar_{\calZ(C_2(g',g''))}[ddd]&\calL_g+\calL_{g'g''}\ar[r]\ar[d]&\calL_{g'g''}+\calL_g\ar[r]\ar[d]&\calL_{gg'g''}\ar^{\calZ(C_2(g',g''))}[ddd]\\
&\calL_g+(\calL_{g'}+\calL_{g''})\ar[r]\ar[d]&(\calL_{g'}+\calL_{g''})+\calL_g\ar[d]&\\
&\calL_g+(\calL_{g''}+\calL_{g'})\ar[r]\ar[d]&(\calL_{g''}+\calL_{g'})+\calL_g\ar[d]&\\
\calL_{gg''g'}\ar[r]\ar@/_1.5pc/|{\calZ(C_2(g,g''g'))}[rrr]&\calL_g+\calL_{g''g'}\ar[r]&\calL_{g'g''}+\calL_g\ar[r]&\calL_{gg''g'}
}\]
The following lemma can be checked using the definition of $B\calP$.
\begin{lem}The isomorphism \eqref{comm=comm} is the same as the commutativity constraint in $\calP$.
\end{lem}
Now, there are two isomorphisms between $(C_2(g,g')+C_2(g,g''))+C_2(g',g'')$ and $C_2(g',g'')+(C_2(g,g'')+C_2(g,g'))$. Namely, the first isomorphism is obtained by the associativity and commutativity constraints in $\calP$. (Recall that such isomorphism is unique by the "Mac Lane's
coherence theorem" for Picard category.) The second isomorphism is
\begin{multline}\label{2nd isom}(C_2(g,g')+C_2(g,g''))+C_2(g',g'')\simeq C_2(g,g'g'')+C_2(g',g'')\stackrel{\eqref{comm=comm}}{\simeq}\\
C_2(g',g'')+C_2(g,g''g')\simeq
C_2(g',g'')+(C_2(g,g'')+C_2(g,g')).\end{multline} By the lemma, the
difference between these two isomorphisms is $C_3(g,g',g'')$. If we
recall the definition of $C_2(g,g'g'')\simeq C_2(g,g')+C_2(g,g'')$
by \eqref{constrC2}, we see that the isomorphism \eqref{2nd isom}
can be represented by the following diagram
\[
\tiny
\xymatrix@C=15pt@R=20pt{&\calL_{g'gg''}\ar^{\calZ(C_2(g,g''))}[rrr]\ar@{=>}[dr]&&&\calL_{g'g''g}\ar[ddr]^{\calZ(C_2(g',g''))}\ar@{=>}[dl]&\\
&&\calL_{g'}+\calL_g+\calL_{g''}\ar[r]&\calL_{g'}+\calL_{g''}+\calL_g\ar[rd]&&\\
\calL_{gg'g''}\ar@{=>}[r]\ar^{\calZ(C_2(g,g'))}[uur]\ar_{\calZ(C_2(g',g''))}[ddr]&\calL_{g}+\calL_{g'}+\calL_{g''}\ar[ur]\ar[dr]&&&\calL_{g''}+\calL_{g'}+\calL_g&\calL_{g''g'g}\ar@{=>}[l]\\
&&\calL_{g}+\calL_{g''}+\calL_{g'}\ar[r]&\calL_{g''}+\calL_{g}+\calL_{g'}\ar[ru]&&\\
&\calL_{gg''g'}\ar@{=>}[ur]\ar_{\calZ(C_2(g,g''))}[rrr]&&&\calL_{g''gg'}\ar@{=>}[ul]\ar_{\calZ(C_2(g,g'))}[uur]&
}\]
This diagram clearly implies that $C_3(g,g',g'')=C_3(g',g'',g)$. This together with the fact
 that $C_3(g,g',g'')=-C_3(g,g'',g')$ (because the map $\on{Comm}(C_g^{\calL})$ is anti-symmetric) implies that
  $C_3$ is anti-symmetric.
\end{proof}

\begin{cor} \label{trivial} (i) If two central extensions $\calL$ and $\calL'$ of $G$ by $\calP$
are isomorphic in $H^2(BG,\calP)$,
then $C_3^\calL=C_3^{\calL'}$.

(ii) $C_3^\calP$ is trivial.
\end{cor}

\begin{cor} \label{2Baer}
For any two central extensions $\calL$  and $\calL'$ of $G$ by
$\calP$ we have
$$
C_3^{\calL + \calL' } = C_3^{\calL} + C_3^{\calL'}  \mbox{.}
$$
\end{cor}
\begin{proof} This follows from lemma~\ref{lembaer}, corollary~\ref{Add} and the definition of map $C_3$.
\end{proof}

\begin{rmk}
If $\calP =BA$, where $A$ is an abelian group, then a central
extension $\calL$ of a group $G$ by the Picard groupoid $\calP$ is a
$gr$-category such that these $gr$-categories are classified by the
group $H^3(G,A)$ with the trivial $G$-module $A$. In this case the
map $C_3^{\calL}$ coincides with the symmetrization of corresponding
$3$-cocycle, see~\cite[\S 4]{Br2}. (It follows from
remarks~\ref{Del},~\ref{usual} and~\cite[prop.~10]{O2}.)
\end{rmk}

\section{Tate vector spaces} \label{tate}
\subsection{The category of Tate vector spaces} \label{tvecs}
We first review the definition of Tate vector spaces, following
\cite{O1,AK}. Let us fix a base field $k$.

Let us recall that for an exact category $\calE$ in the sense of Quillen (cf. \cite{Q}), Beilinson associates another exact category $\lim\limits_{\longleftrightarrow}\calE$, which is again an exact category (cf. \cite{Be}). In nowadays terminology, this is the category of locally compact objects of $\calE$.

For an exact category $\calE$, let $\hat{\calE}$ denote the category
of left exact additive contravariant functors from the category
$\calE$ to the category of abelian groups. This is again an exact
category (in fact an abelian category), in which arbitrary small
colimits exist. The Yoneda embedding $h:\calE\to\hat{\calE}$ is
exact. Then the category $\on{Ind}(\calE)$ of (strict) ind-objects
of $\calE$, is the full subcategory of $\hat{\calE}$ consisting of
objects of the form $\mathop{\underrightarrow{\lim}}\limits_{i\in
I}h(X_i)$, where $I$ is a filtered small category, and
$X_i\in\calE$, such that for $i\to j$ in $I$, the map $X_i\to X_j$
is an admissible monomorphism. This category is a natural exact
category. Likewise, one can define $\on{Pro}(\calE)$ as
$\on{Ind}(\calE^{op})^{op}$.

\begin{dfn}Let $\calE$ be an exact category. Then $\lim\limits_{\longleftrightarrow}\calE$ is the full subcategory of
$\on{Pro}(\on{Ind}(\calE))$ consisting of objects that can be represented as
 $\mathop{\underleftarrow{\lim}}\limits_{i\in I^{op}}
 \mathop{\underrightarrow{\lim}}\limits_{j\in I}h(X_{ij})$ such that
 for any $i\to i', j\to j'$, the following diagram is cartesian (which is automatically cocartesian then).
\[\xymatrix{X_{ij}\ar[r]\ar[d]&X_{ij'}\ar[d]\\
                   X_{i'j}\ar[r]&X_{i'j'}
}\]
\end{dfn}

One can show that $\lim\limits_{\longleftrightarrow}\calE$ is an exact category and the following embedding
$\lim\limits_{\longleftrightarrow}\calE\to\on{Pro}(\on{Ind}(\calE))$ is exact. In addition, there is a
natural embedding $\lim\limits_{\longleftrightarrow}\calE\to\on{Ind}(\on{Pro}(\calE))$ which is again exact.
 It is clear that the natural embedding $\on{Ind}(\calE)\to\on{Pro}(\on{Ind}(\calE))$ lands
 in $\lim\limits_{\longleftrightarrow}\calE$ and similarly the natural embedding
  $\on{Pro}(\calE)\to\on{Ind}(\on{Pro}(\calE))$ also lands in $\lim\limits_{\longleftrightarrow}\calE$.

\begin{dfn}Define $\Tate_0$ to be the category of finite dimensional vector spaces,
together with its canonical exact category structure.
Define $\Tate_n=\lim\limits_{\longleftrightarrow}\Tate_{n-1}$, together with the canonical
 exact category structure given by Beilinson.
\end{dfn}

There is a canonical forgetful functor $F_n:\Tate_n\to \calT op$, where $\calT op$ denotes the category of
 topological vector spaces. As is shown in \cite{O1}, the functor is fully faithful when $n=1$,
  but this is in general not the case when $n>1$.

\begin{dfn} \label{latcol} Let $V$ be an object of $\Tate_n$. A {\em lattice} $L$ of $V$ is an object in $\Tate_n$
 which  actually belongs to $\on{Pro}(\Tate_{n-1})$, together with an admissible monomorphism $L\to V$
 such that the object $V/L$ belongs to $\on{Ind}(\Tate_{n-1})$.  A {\em colattice} $L^c$ of $V$ is an object in
 $\Tate_n$
   which actually belongs to $\on{Ind}(\Tate_{n-1})$, together with an addmissible monomorphism $L^c\to V$
   such that the object $V/L^c$ belongs to $\on{Pro}(\Tate_{n-1})$.
\end{dfn}

It is clear that if $L$ is a lattice of $V$ and $L^c$ is a colattice, then $L\cap L^c$ belongs to $\Tate_{n-1}$.

The main players of this paper are $\Tate_1$ and $\Tate_2$. The
category $\Tate_1$ is just the category of locally linearly compact
$k$-vector spaces. A typical object in $\Tate_1$ is the field of
formal Laurent series $k((t))$, i.e. $k((t))$ is the field of
fractions of the ring $k[[t]]$. $k((t))$ is equipped with the
standard  topology, where the base of neighbourhoods of zero
consists of integer powers of the maximal ideal of $k[[t]]$. The
subspace $k[[t]]$ is a lattice in $k((t))$ and $k[t^{-1}]$ is a
colattice. Observe that $k[t]\subset k((t))$ is neither a lattice
nor a colattice, because the subspace $k[t]$ is not closed in the
topological space $k((t))$. Therefore the embedding $k[t]
\hookrightarrow k((t))$ is not an admissible monomorphism, since any
admissible (exact) triple in the category $\Tate_1$ is of the form:
$$
0  \lrto V_1 \lrto V_2  \lrto V_3 \lrto 0 \mbox{,}
$$
where the locally linearly compact vector space $V_1$ is a closed
vector subspace in a locally linearly compact vector space $V_2$,
and the locally linearly compact vector space $V_3$ has the quotient
topology on the quotient vector space.

A typical object in $\Tate_2$ is $k((t))((s))$, since
$$
k((t))((s)) =\mathop{\underleftarrow{\lim}}\limits_{l \in \bbZ}
 \mathop{\underrightarrow{\lim}}\limits_{m \le l} s^m k((t))[[s]]
 \,/\,
 s^l k((t))[[s]] =
\mathop{\underrightarrow{\lim}}\limits_{m \in \bbZ}
\mathop{\underleftarrow{\lim}}\limits_{l \ge m }
 s^m k((t))[[s]] \,/ \,
 s^l k((t))[[s]] \mbox{,}
$$
and $s^m k((t))[[s]] \,/\,
 s^l k((t))[[s]]$ is a locally linearly compact $k$-vector space.

The $k$-space $k((t))[[s]]$ is a lattice, and the $k$-space
$k((t))[s^{-1}]$ is a colattice in the $k$-space $k((t))((s))$. As
just mentioned above, it is not enough to  regard them as
topological vector spaces. On the other hand $k[[t]]((s))$ is not a
lattice in $k((t))((s))$ although the natural map $k[[t]]((s))\to
k((t))((s))$ is an admissible monomorphism.

\begin{rmk}
The category $\Tate_n$ coincides with  the category of  complete $C_n$-spaces from~\cite{O1}.
\end{rmk}

\subsection{Determinant theories of Tate vector spaces}
\label{detth}

We consider $\Tate_0$ as an exact category. Then $\det: (\Tate_0,
\on{isom} )\to\calP ic^\bbZ$  (see~\eqref{det}) is a functor
satisfying the following additional  property: for each injective
homomorphism $V_1\to V$ in the category $\Tate_0$, there is a
canonical isomorphism
\begin{equation}\label{tensor}\det(V_1)\otimes\det(V/V_1)\simeq\det(V) \mbox{,}\end{equation}
such that:

(i) for $V_1=0$ (resp. $V_1=V$), equality \eqref{tensor} is the same as
\begin{equation}\label{unit1}\ell_0\otimes\det(V)\simeq\det(V)\end{equation}
resp.
\begin{equation}\label{unit2}\det(V)\otimes\ell_0\simeq\det(V) \mbox{,}\end{equation}
where $\ell_0$ is the trivial $k$-line of degree zero.

(ii) For any diagram
\begin{equation}\label{2x3}
\xymatrix{0\ar[r]&U_1\ar[r]\ar^\simeq[d]&U\ar[r]\ar^\simeq[d]&U/U_1\ar[r]\ar^\simeq[d]&0\\
0\ar[r]&V_1\ar[r]&V\ar[r]&V/V_1\ar[r]&0,
}
\end{equation}
the following diagram is commutative
\begin{equation}\label{diag1}
\xymatrix{
\det(U_1)\otimes\det(U/U_1)\ar[r]\ar[d]&\det(U)\ar[d]\\
\det(V_1)\otimes\det(V/V_1)\ar[r]&\det(V).}
\end{equation}

(iii) For any diagram
\begin{equation}\label{3x3}
\xymatrix{&0\ar[d]&0\ar[d]&0\ar[d]&\\
0\ar[r]&U_1\ar[r]\ar[d]&V_1\ar[r]\ar[d]&W_1\ar[r]\ar[d]&0\\
0\ar[r]&U\ar[r]\ar[d]&V\ar[r]\ar[d]&W\ar[r]\ar[d]&0\\
0\ar[r]&U/U_1\ar[r]\ar[d]&V/V_1\ar[r]\ar[d]&W/W_1\ar[r]\ar[d]&0\\
&0&0&0&,}\end{equation} the following diagram is commutative
\begin{equation}\label{diag2}
\xymatrix{(\det(U_1)\otimes\det(U/U_1))\otimes(\det(W_1)\otimes\det(W/W_1))\ar[rr]\ar^{ass.\ and\ comm.\ constraints}[d]&&\det(U)\otimes\det(W)\ar[dd]\\
(\det(U_1)\otimes\det(W_1))\otimes(\det(U/U_1)\otimes\det(W/W_1))\ar[d]&&\\
\det(V_1)\otimes\det(V/V_1)\ar[rr]&&\det(V)}\end{equation}

\begin{dfn}Let $\calP$ be a Picard groupoid. A determinant functor from the category $(\Tate_0, \on{isom})$ to
$\calP$ is a functor $D: (\Tate_0, \on{isom})\to\calP$ together with isomorphisms \eqref{tensor} satisfying equalities and diagrams  \eqref{unit1}-\eqref{diag2}, where we have to change  notation $"\det"$
to notation $"D"$ everywhere in these formulas.
\end{dfn}

There is the following obvious proposition.
\begin{prop}Let $D:(\Tate_0, \on{isom} )\to\calP$ be a determinant functor. Then there is a $1$-homomorphism of Picard groupoids $\tilde{D}:\calP ic^\bbZ\to\calP$  and a monoidal natural transformation $\varepsilon:\tilde{D}\circ\det\simeq D$. Furthermore, the pair $(\tilde{D},\varepsilon)$ is unique up to a unique isomorphism.
\end{prop}

\begin{rmk}All the above discussions are valid when one replaces $k$ by a noetherian commutative ring $A$, and replaces $\Tate_0$ by the category of finitely generated projective $A$-modules.
\end{rmk}

Next we turn to $\Tate_1$. The following result is fundamental and
is due to Kapranov, \cite{Kap}, but see also~\cite[\S 5.1-5.3]{Dr}.
\begin{prop}\label{det gerbe} There is a natural functor
\[\Det \: : \: (\Tate_1, \on{isom}) \to B\calP ic^\bbZ,\]
and for each admissible monomorphism $\calV_1\to \calV$ there is a
$1$-isomorphism
\begin{equation} \label{aaa}
\Det(\calV_1)+\Det(\calV/\calV_1)\to \Det(\calV)
\end{equation}
such that if $\calV_1=0$ (resp. $\calV_1=\calV$), this
$1$-isomorphism is the canonical $1$-isomorphism
\[\calP+\calD et(\calV)\simeq\calD et(\calV)\]
resp.
\[\calD et(\calV)+\calP\simeq\calD et(\calV) \mbox{.}\]
For each admissible diagram \eqref{2x3} of $1$-Tate vector spaces,
the corresponding diagram  \eqref{diag1} is commutative. In
addition, for each admissible diagram \eqref{3x3} of $1$-Tate vector
spaces, there is a $2$-isomorphism for the corresponding
diagram~\eqref{diag2}.
\end{prop}
\begin{rmk} Under conditions of proposition~\ref{det gerbe}, the $2$-isomorphisms which appear from
diagram~\eqref{diag2} satisfy further compatibility conditions.
\end{rmk}
\begin{proof}  For a $1$-Tate vector space $\calV$, we recall the definition of a graded-determinantal theory
 $\Delta$ on $\calV$. This is a rule that assign to every lattice $L \subset \calV$ an object $\Delta(L)$ from
 $\calP ic^\bbZ$ and to every lattices $L_1 \subset L_2 \subset \calV$ an isomorphism
$$
\Delta_{L_1, L_2} \: : \: \Delta(L_1) \otimes \det(L_2/L_1) \lrto \Delta(L_2)
$$
such that for any three lattices $L_1 \subset L_2 \subset L_3
\subset \calV$ the following diagram is commutative
$$
\xymatrix{
\Delta(L_1) \otimes \det(L_2/L_1) \otimes \det(L_3/L_2)\ar[r]\ar[d]& \Delta(L_1) \otimes \det(L_3/L_1)\ar[d]\\
\Delta(L_2) \otimes \det(L_3/L_2)\ar[r]& \Delta(L_3) \mbox{.}
}
$$
Let $\calD et(\calV)$ be the category of graded-determinantal
theories on $\calV$. This is a $\calP ic^\bbZ$-torsor, where for any
$x \in \calP ic^\bbZ$, $\Delta \in \calD et(\calV)$, we have
$(x+\Delta) (L) := x \otimes \Delta(L)$.

Now for an admissible (exact) sequence
$$
0 \lrto \calV_1 \lrto \calV \stackrel{\epsilon}{\lrto} \calV/\calV_1
\lrto 0
$$
the  $1$-isomorphism~\eqref{aaa}  is constructed as
 $$
\Delta(L) := \Delta_1(L \cap \calV_1) \otimes \Delta_2(\epsilon(L))
\mbox{,}
 $$
where $L$ is a lattice in $\calV$, $\Delta_1 \in \calD et(\calV_1)$,
$\Delta_2 \in \calD et(\calV/\calV_1)$, $\Delta \in \calD
et(\calV)$. (We used that the $k$-space $L \cap \calV_1$ is a
lattice in the $1$-Tate vector space $\calV_1$, and the $k$-space
$\epsilon(L)$ is a lattice in the $1$-Tate vector space $\calV_3$).

We note that, by construction,  $\calV \mapsto \calD et(\calV)$ is
naturally a contravariant functor from the category $(\Tate_1,
\on{isom})$ to the category $B\calP ic^\bbZ$. To obtain the
covariant functor we have to inverse arrows in the category
$(\Tate_1, \on{isom})$.
\end{proof}

\section{Applications to the case $G=\GL (k((t)))$ and $\GL (k((t))((s)))$}
\label{appl}
\subsection{Tame symbols} \label{ts}
Let us first review the tame symbols. Recall that if $K$ is a field
with discrete valuation $\nu:K^\times\to\bbZ$, and  $k$ denote its
residue field, then there are  so-called boundary maps for any $i
\in \mathbb{N}$
\[\partial_i \: : \:K^M_i(K) \lrto K^M_{i-1}(k),\]
where $K^M_i(F)$ denotes the $i$th Milnor K-group of a field $F$. Let us also recall that for a field $F$, the $i$th Milnor K-group $K^M_i(F)$ is the quotient of the abelian group $F^\times\otimes_\bbZ F^\times\otimes_\bbZ\cdots\otimes_\bbZ F^\times$ modulo the so-called Steinberg relations. Then the tame symbol is defined as the composition of the following maps
\[\{\cdot,\cdot\} \: : \:K^\times\otimes_\bbZ K^\times\lrto K^M_2(K)\stackrel{\partial_2}{\lrto} K_1^M(k)\simeq k^\times.\]
Explicitly, let $\pi\subset K$ be the maximal ideal. Then
\begin{equation}\label{ts2}
\{f,g\}=(-1)^{\nu(f)\nu(g)}\frac{f^{\nu(g)}}{g^{\nu(f)}} \mod \pi
\end{equation}

Now, let $\bbK$ be a two-dimensional local field, whose residue field is denoted by $K$, whose residue field is $k$. Then we define the following map as
\[\nu_\bbK \: : \:\bbK^\times\otimes_\bbZ\bbK^\times\lrto K_2^M(\bbK) \stackrel{\partial_2}{\lrto} K_1^M(K) \stackrel{\partial_1}{\lrto} K_0^M(k)\simeq\bbZ \mbox{,}\]
and define the two-dimensional tame symbol as
\[\{\cdot,\cdot,\cdot\} \: : \: \bbK^\times\otimes_\bbZ\bbK^\times\otimes_\bbZ\bbK^\times\lrto K_3^M(\bbK) \stackrel{\partial_3}{\lrto} K_2^M(K)
\stackrel{\partial_2}{\lrto} K_1^M(k)\simeq k^\times.\] We have the
following explicit formulas for $\nu_\bbK$ and
$\{\cdot,\cdot,\cdot\}$ (see \cite{O2}.) Let $\nu_1:\bbK\to\bbZ$,
and $\nu_2:K\to\bbZ$ be discrete valuations. Let $\pi_\bbK$ be the
maximal ideal of $\bbK$, $\pi_K$ be the maximal ideal of $K$. For an
element $f\in\calO_\bbK$, let $\bar{f}$ denote its residue class in
$K$. Then
\begin{equation}\label{nuk}
\nu_\bbK(f,g)=\nu_2(\overline{\frac{f^{\nu_1(g)}}{g^{\nu_1(f)}}})
\end{equation}
\begin{equation}\label{ts3}
\{f,g,h\}=\on{sgn}(f,g,h)f^{\nu_\bbK(g,h)}g^{\nu_\bbK(h,f)}h^{\nu_\bbK(f,g)} \mod \pi_\bbK \mod \pi_K
\end{equation}
where
\begin{equation} \hspace{-0.2cm} \label{sgn}
\begin{array}{c}\on{sgn}(f,g,h)=(-1)^A,\\
A=\nu_\bbK(f,g)\nu_\bbK(f,h)+\nu_\bbK(g,h)\nu_\bbK(g,f)+\nu_\bbK(h,f)\nu_\bbK(h,g)+\nu_\bbK(f,g)\nu_\bbK(g,h)\nu_\bbK(h,f).
\end{array}
\end{equation}
\begin{rmk}Originally one used another explicit formula for the sign of
the two-dimensional tame symbol.
 This other formula was introduced in~\cite{Pa0}.
\end{rmk}

It is easy to see that tame symbols $\{\cdot, \cdot \}$, $\{ \cdot,
\cdot, \cdot \}$ and the map $\nu_\bbK$ are anti-symmetric.

\subsection{The one-dimensional story} \label{one-story}
Let $\calV$ be a $1$-Tate vector space over $k$. The group of
automorphisms of $\calV$ in this category is denoted by
$\GL(\calV)$.
\begin{prop}\label{DetV}
There is a homomorphism $\Det_\calV:\GL(\calV)\to\calP ic^\bbZ$,
which is canonical up to a unique isomorphism in
$H^1(B\GL(\calV),\calP ic^\bbZ)$.
\end{prop}
\begin{proof}According to proposition~\ref{det gerbe}, we have a
homomorphism
\[\GL(\calV)\to\Hom_{\calP ic^\bbZ}(\calD et(\calV),\calD et(\calV))\simeq\calP ic^\bbZ\]
via $\calZ^{-1}$, where $\calZ:\calP ic^\bbZ\to\Hom_{\calP
ic^\bbZ}(\calD et(\calV),\calD et(\calV))$ is a natural homomorphism
from section~\ref{Ptorsors}.
\end{proof}

Choose $\calL\subset \calV$ a lattice. It follows from the proof of
proposition~\ref{det gerbe} that in concrete terms, one has to
assign to $\Det_\calV(g)$ the graded line
\begin{equation} \label{grline}
\det(\calL \mid g\calL):=\det\left(\frac{g\calL}{\calL\cap g\calL}\right)\otimes \,
\det\left(\frac{\calL}{\calL\cap g\calL}\right)^{-1} \mbox{,}
\end{equation}
where $g \in \GL(\calV)$.
Then, it is well-known that there is a canonical isomorphism
\[\det(\calL \mid gg'\calL)\simeq\det(\calL|g\calL)\otimes\det(g\calL \mid gg'\calL)\simeq
\det(\calL\mid g\calL)\otimes\det(\calL \mid g'\calL) \mbox{,}\]
which is compatible with the associativity constraints in the
category $\calP ic^\bbZ$ (see, for example, \cite[\S 1]{FZ}). For
different choice of $\calL$, the resulting objects in
$H^1(BGL(\calV),\calP ic^\bbZ)$ are isomorphic.

We also have the following lemma, which easily follows from the
construction of homomorphism $\Det_\calV$ and the discussion in \S
\ref{detth} (in particular the diagram \eqref{diag2}).
\begin{lem}\label{fil}
If $0\to\calV'\to\calV\to\calV''\to 0$ is a short exact sequence of $1$-Tate vector spaces
(recall that $\Tate_1$ is an exact category). Let $P$ be the
subgroup of $\GL(\calV)$ that preserves this sequence, then there is
a canonical $1$-isomorphism $\calD et_{\calV'}+\calD
et_{\calV''}\simeq\Det_\calV$  in $H^1(BP,\calP ic^\bbZ)$.
\end{lem}
\begin{rmk}The 1-homomorphism $F_{\calP ic}\circ\calD et_\calV:
\GL(\calV)\to\calP ic$ is essentially constructed in \cite{ACK}.
However, the above lemma does not hold for this 1-homomorphism. This
is the complication of the sign issues in \cite{ACK}.
\end{rmk}

Now let $k'/k$ be a finite extension and $K=k'((t))$ be a local
field with residue field $k'$. Then $K$ has a natural structure as a
$1$-Tate vector space over $k$. Let $H=K^\times$.
 The multiplication gives a natural embedding $H\subset\GL(K)$. The
 following proposition is from~\cite{BBE}.
\begin{prop}\label{tame symbol} If $f,g\in H$, then
\[\on{Comm}(\calD et_{K})(f,g)=\on{Nm}_{k'/k}\{f,g\}^{-1} \]
that is inverse to the tame symbol of $f$ and $g$.
\end{prop}

\begin{rmk}
Since the natural functor $F_{\calP ic}$ is monoidal, the
restriction to $H$ of the functor  $ F_{\calP ic} \circ \calD
et_{K}$ determines a homomorphism $H \to \calP ic$. The commutator
pairing $\on{Comm}(f,g)$ constructed by this homomorphism is
$$(-1)^{\on{ord}(f)\on{ord}(g)}\on{Nm}_{k'/k}\{f,g\}^{-1} \mbox{.}$$
\end{rmk}

By definition~\ref{latcol}, a lattice $\calL$ of $\calV$ is a
linearly compact open $k$-subspace of $\calV$ such that $\calV
/\calL$ is a discrete $k$-space. A colattice $\calL^c$ is a
$k$-subspace of $\calV$ such that for any lattice $\calL$, both
$\calL^c\cap\calL$ and $\calV/(\calL^c+\calL)$ are finite
dimensional.

\begin{lem}\label{prop3.4}
Let $P\subset \GL(\calV)$ be a subgroup of $\GL(\calV)$ that preserves a lattice (or a colattice)
in $\calV$, then the homomorphism $\calD et_\calV$ is trivial on
$P$.
\end{lem}
\begin{proof}
Let $\calL \subset \calV$ be a lattice such that the group $P$
preseves it. We consider an exact sequence of $1$-Tate vector spaces
$$
0 \lrto \calL \lrto \calV  \lrto \calV / \calL \lrto 0 \mbox{.}
$$
Then the group $P$ preserves this sequence. Therefore by
lemma~\ref{fil}, it is enough to prove that the homomorphisms $\calD
et_\calL$ and $\calD et_{\calL / \calV}$ are trivial on $P$. But
this is obvious from the proof of proposition~\ref{DetV}.

For a colattice $\calL^c \subset \calV$ we have to use the analogous
reasonings.
\end{proof}

\subsection{The two-dimensional story}  \label{2-story}

If $\bbV\in\Tate_2$, then we denote by $\GL(\bbV)$  the group of
automorphisms of $\bbV$ in this category.

There should be a determinantal functor from $(\Tate_2, \on{isom})$
to $B^2\calP ic^\bbZ$, which assigns to every such $ \bbV$ the
graded gerbel theory in the sense of \cite{AK}, satisfying
properties which generalize properties listed in
proposition~\ref{det gerbe} (and further compatibility conditions).
We do not make it precise. But we define the corresponding central
extension of $\GL(\bbV)$ as follows. Pick  a lattice $\bbL$ of
$\bbV$. Then one associates with $g$ the $\calP ic^\bbZ$-torsor
\begin{equation} \label{pictorsor}
\Det_\bbV(g)=\Det(\bbL \mid g\bbL):=\Det \left(\frac{g\bbL}{\bbL\cap g\bbL}
\right) -\Det \left(\frac{\bbL}{\bbL\cap g\bbL} \right) \mbox{.}
\end{equation}
This definition is correct because both $k$-spaces
$\frac{g\bbL}{\bbL\cap g\bbL}$ and $\frac{\bbL}{\bbL\cap g\bbL}$
belong to objects of  category $\Tate_1$. We define the
$1$-isomorphism as
\begin{equation} \label{deiso}
\Det(\bbL \mid gg'\bbL)\simeq\Det(\bbL \mid g\bbL)+\Det(g\bbL \mid
gg'\bbL)\simeq \Det(\bbL \mid g\bbL)+\Det(\bbL \mid g'\bbL) \mbox{.}
\end{equation}
One uses proposition~\ref{det gerbe} to check that this defines a
central extension of $\GL(\bbV)$ by $\calP ic^\bbZ$. This central
extension depends on the chosen lattice $\bbL$ of $\bbV$. If we
change the lattice, then the central extension constructed by a new
lattice will be isomorphic to the previous one.
\begin{rmk}If
one replaces $\calP ic^\bbZ$ by $\calP ic$, such a central extension
was constructed in~\cite{O2, FZ}. Besides, in~\cite{O2} the
two-dimensional tame symbol up to sign was obtained as an
application of this construction, and the reciprocity laws on
algebraic surfaces were proved up to sign.
\end{rmk}

As generalization of lemma ~\ref{fil} and lemma~\ref{prop3.4} it is
not difficult to prove the following lemmas.
\begin{lem}\label{fil2}
If $0\to\bbV'\to\bbV\to\bbV''\to 0$ is a short exact sequence of $2$-Tate vector spaces
(recall that $\Tate_2$ is an exact category). Let $P$ be the
subgroup of $\GL(\bbV)$ that preserves this sequence, then there is
a canonical $1$-isomorphism $\calD et_{\bbV'}+\calD
et_{\bbV''}\simeq\Det_\bbV$  in $H^2(BP,\calP ic^\bbZ)$.
\end{lem}

\begin{lem}\label{t} Let $P$ be subgroup of $\GL(\bbV)$ which preserves a lattice or a colattice in $\bbV$,
then the central extension restricted to $P$ can be trivialized.
\end{lem}

Let $k'/k$ be a finite field extension, and $\bbK=k'((t))((s))$ be a
two-dimensional local field. Then $\bbK$ has a natural structure as
a $2$-Tate vector space over $k$. The group $H=\bbK^\times$ acts on
$\bbK$ by left multiplications, which gives rise to an embedding
$H\to\GL(\bbK)$.

\begin{thm} \label{2-tame} For $f,g,h\in H$, one has
\[C^{\Det}_3(f,g,h)=\on{Nm}_{k'/k}\{f,g,h\},\]
where the map $C^\Det_3$ is  constructed in lemma~\ref{C3} and
$\{\cdot,\cdot,\cdot\}$ is  the two-dimensional tame symbol.
\end{thm}
In what follows, we will denote the bimultiplicative homomorphism
$C_2^{\calD et}$ by $C_2$, the homomorphism $C_g^{\calD et}$ by
$C_g$ and the map $C_3^{\calD et}$ by $C_3$.
\begin{proof}Since both maps $C_3$ and $\on{Nm}_{k'/k}\{\cdot,\cdot,\cdot\}$ are anti-symmetric and
tri-mul\-ti\-pli\-ca\-ti\-ve, we just need to consider the following
cases: (i) $f,g,h\in\calO_{\bbK}^\times$; (ii)
$f,g\in\calO_{\bbK}^\times,h=s$; (iii)
$f\in\calO_{\bbK}^\times,g=h=s$; (iv) $f=g=h=s$. Here
$\calO_\bbK=k'((t))[[s]]$ is the ring of integers of the field
$\bbK$, which is also a lattice in $\bbK$. We will fix $\bbL
=\calO_\bbK$.

In the first case, we have that both $C_3$ and
$\on{Nm}_{k'/k}\{\cdot,\cdot,\cdot\}$ are trivial (to see that $C_3$
is trivial, one uses lemma~\ref{t}).

\medskip

According to formulas~\eqref{nuk}-\eqref{sgn}, the second case
amounts to proving that
\[C_3(f,g,s)=\on{Nm}_{k'/k}\{\bar{f},\bar{g}\},\]
where $\bar{f},\bar{g}$ are the image of elements $f,g$ under the
map $\calO_\bbK^\times\to K^\times$.

Let us consider a little more general situation. Let
$f,g\in\GL(\bbK)$ that leave the lattice $\calO_\bbK$ invariant, and
let $h\in\GL(\bbK)$ such that $h\calO_\bbK\subset\calO_\bbK$. Let
$\calV=\calO_\bbK/h\calO_\bbK$, which is a $1$-Tate vector space
over the field $k$. We assume that $f,g,h$ mutually commute with
each other. Then $f,g:\calO_\bbK\to\calO_\bbK$ induce automorphisms
$\pi_h(f),\pi_h(g):\calV\to\calV$.
 Let $\calD et$ be the central extension of $\GL(\bbK)$ by $\calP ic^\bbZ$ defined by the lattice $\bbL=\calO_\bbK$.
 By definition,
 under the isomorphism $\calZ:\calP ic^\bbZ\to\Hom_{\calP ic^\bbZ}(\calD et(\calO_\bbK|hg\calO_\bbK),\calD
 et(\calO_\bbK|hg\calO_\bbK))$, the $1$-isomorphism $C_2(h,g)$ corresponds
 to the following composition of $1$-isomorphisms of $\calP
 ic^\bbZ$-torsors
\[\small\begin{split}\calD et(\calO_\bbK | hg\calO_\bbK)\to\calD et(\calO_\bbK | h\calO_\bbK)+\calD
et(h\calO_\bbK | hg\calO_\bbK)\to\calD et(\calO_\bbK | h\calO_\bbK)+\calD et(\calO_\bbK | g\calO_\bbK)\\
\to\calD et(\calO_\bbK | g\calO_\bbK)+\calD et(\calO_\bbK |
h\calO_\bbK)\to\calD et(\calO_\bbK | g\calO_\bbK)+\calD
 et(g\calO_\bbK | gh\calO_\bbK)\to\calD et(\calO_\bbK |
 gh\calO_\bbK) \mbox{.}
\end{split}\]

Using the fact that $g\calO_\bbK=\calO_\bbK$ and
proposition~\ref{det gerbe}, the above $1$-isomorphism is
canonically $2$-isomorphic to the following $1$-isomorphism
\[\calD et(\calO_\bbK \mid hg\calO_\bbK)\stackrel{\calZ(\calD et_{\calV}(\pi(g)))}{\longrightarrow}
\calD et(\calO_\bbK \mid hg\calO_\bbK).\] Therefore, there is a
canonical $2$-isomorphism $C_2(h,g)\simeq -\calD
et_\calV(\pi_h(g))$, because, by definition (see
formula~\eqref{pictorsor}), $\calD et(\calO_\bbK \mid h\calO_\bbK)
\simeq - \calD et(\calV)$.
 One readily checks by the
construction of lemma-definition~\ref{C2}, that these
$2$-isomorphisms fit into the following commutative diagrams
\begin{equation}\label{I}
\xymatrix{C_2(h,fg)\ar[r]\ar[d]&C_2(h,f)+C_2(h,g)\ar[d]\\
-\Det_\calV(\pi_h(fg))\ar[r]&-\Det_\calV(\pi_h(f))-\Det_\calV(\pi_h(g))
}\end{equation} where the natural isomorphism
$\Det_\calV(\pi_h(fg))\to\Det_\calV(\pi_h(f))+\Det_\calV(\pi_h(g))$
comes from proposition~\ref{DetV}. (We have to use that $\calD
et((0))$ is canonically isomorphic to $\calP ic^{\bbZ}$, and
$\calO_\bbK/ g\calO_\bbK= (0)$, where $(0)$ is the zero-space. )

We now return to our proof of case (ii). Let $P_s$ be the subgroup
of $\GL(\bbK)$ consisting of elements that preserve the lattice
$\calO_\bbK$ and commute with the element $s$. Then elements in the
group $P_s$ also preserve the lattice $s\calO_\bbK$, and therefore
induce a group homomorphism
\[\pi_s :P_s\to \GL(K),\]
because $\bbK = \calO_\bbK / s \calO_\bbK$. Then the commutative
diagram \eqref{I} amounts to the following lemma.

\begin{lem}The homomorphism $C_s:P_s\to\calP ic^\bbZ$ is isomorphic to
the minus (or the inverse) of the following homomorphism
\[\Det_K\circ\pi_s \: : \: P_s\to \GL(K)\to \calP ic^\bbZ.\]
\end{lem}
By proposition~\ref{tame symbol}, we thus obtain that
\[C_3(f,g,s)=C_3(s,f,g)=\Comm(C_s)(f,g)=\on{Nm}_{k'/k}\{\bar{f},\bar{g}\}\] for $f,g\in\calO_\bbK^\times\subset\GL(\bbK)$.
The case (ii) follows.

\medskip

Case (iii). According to formulas~\eqref{nuk}-\eqref{sgn}, one needs
to show
\[C_3(f,s,s)=C_f(s,s)= \on{Nm}_{k'/k} (-1)^{\nu_2(\bar{f})}=
(-1)^{(\nu_2(\bar{f})[k':k])} = (-1)^{(\nu_2(\bar{f})[k':k])^2}
\mbox{.}\] We have the following exact sequence of $1$-Tate vector
spaces
\[0\lrto\frac{s\calO_\bbK}{s^2\calO_\bbK}\lrto\frac{\calO_\bbK}{s^2\calO_\bbK}\lrto\frac{\calO_\bbK}{s\calO_\bbK}
\lrto 0 \mbox{.}\] and therefore by lemma~\ref{fil}, for any element
$p \in P_s$, there is a canonical isomorphism in $\calP ic^\bbZ$
\begin{equation}\label{x}\calD et_{\frac{\calO_\bbK}{s^2\calO_\bbK}}(\pi_{s^2}(p))
\simeq\calD
et_{\frac{s\calO_\bbK}{s^2\calO_\bbK}}(\pi_{s^2}(p))+\calD
et_{\frac{\calO_\bbK}{s\calO_\bbK}}(\pi_{s^2}(p))\mbox{.}
\end{equation}
On the other hand, we have already shown that there are canonical isomorphisms
\begin{equation}\label{y}C_2(s,p)\simeq -\calD
et_{\frac{\calO_\bbK}{s\calO_\bbK}}(\pi_s(p)) =-\calD
et_{\frac{\calO_\bbK}{s\calO_\bbK}}(\pi_{s^2}(p))  \mbox{,} \qquad
C_2(s^2,p)\simeq-\calD
et_{\frac{\calO_\bbK}{s^2\calO_\bbK}}(\pi_{s^2}(g)) \mbox{.}
\end{equation}
Again, by checking the construction as in lemma-definition~\ref{C2},
one obtains that under the isomorphisms~\eqref{y}, the canonical
isomorphism $C_2(s^2,p)\simeq C_2(s,p)+C_2(s,p)$ corresponds to
\eqref{x}.

Now let $p=f$ as in the Case (iii). We know that $\calD
et_{\frac{\calO_\bbK}{s\calO_\bbK}}(\pi(f))$ is a graded line of
degree $\nu_2(\bar{f})[k':k]$. Therefore, using $C_2(a,b) \simeq    -C_2(b,a)$ for any commuting elements  $a,b \in \GL(\bbK)$, we
obtain that Case (iii) follows from the definition of the
commutativity constraints in $\calP ic^\bbZ$.

\medskip

Case (iv). One needs to show that $C_s(s,s)=1$. One can easily show
that there are canonical isomorphisms $C_2(s,s)\simeq\ell_0,
C_2(s^2,s)\simeq\ell_0$, and the canonical isomorphism
$C_2(s^2,s)\simeq C_2(s,s)+C_2(s,s)$ corresponds to $\ell_0\simeq   \ell_0+\ell_0$.
(We used that for the $k'$-space $M= k'[[t]]((s))$
we have $s M =M$, and the $k'$-space $M$ induce a lattice in every
$1$-Tate vector space $s^n \calO_\bbK / s^{n+l} \calO_\bbK$, $n \in
\bbZ$, $l \in \bbN$.) This case also follows.
\end{proof}

\section{Reciprocity laws} \label{reslaws}
We will use the ad\`ele theory on schemes.  Ad\`eles on algebraic
surfaces were introduced by Parshin in~\cite{Pa1}.  On arbitrary
noetherian schemes they were considered by Beilinson in~\cite{Be1}.
See the proof of part of results of~\cite{Be1} in~\cite{H}. A survey
of ad\`eles can be found in~\cite{O4}.

We fix a perfect field $k$.

\subsection{Weil reciprocity law} \label{Weil}
To fix the idea, let us first revisit the Weil reciprocity law.
Let $C$ be an irreducible  projective curve over a field $k$. Let
$k(C)$ be the field of rational functions on the curve $C$. For a
closed point
 $p
\in C$ let $\hat{\oo}_p$ be the completion by maximal ideal $m_p$ of
the local ring $\oo_p$ of point $p \in C$. Let a ring $K_p$ be the
localization of the ring $\hat{\oo}_p$ with respect to the
multiplicative system $\oo_p \setminus 0$. (If $p $ is a smooth
point, then $K_p=k(C)_p$ is the fraction field of the ring
$\hat{\oo}_p$, and $K_p = k(p)((t_p))$, $\hat{\oo}_p = k(p)[[t_p]]$,
where $k(p)$ is the residue field of the point $p$, $t_p$ is a local
parameter at $p$. For a non-smooth point $p \in C$, the ring $K_p$
is a finite direct product of one-dimensional local fields.)

We have that $K_p$ is a $1$-Tate vector space over $k$, and
$\hat{\oo}_p$ is a lattice in $K_p$ for any point $p \in C$.

 For any coherent subsheaf
 $\ff$
 of the constant sheaf $k(C)$ on the
curve $C$ we consider the following ad\`ele complex $\ad_C(\ff)$:
$$
\da_{C,0}(\ff) \oplus \da_{C,1}(\ff) \lrto \da_{C,01}(\ff)
$$
whose cohomology groups coincide with the cohomology groups
$H^*(C,\ff)$. Let us recall that
$$
\da_{C, 0} (\ff) = k(C)\otimes_{\oo_C} \ff \mbox{,} \qquad \da_{C,
1} (\ff)=\prod_{p \in C} \hat{\oo}_p \otimes_{\oo_C} \ff \mbox{,}
$$
$$
\da_{C, 01}(\ff) = \da_C = {\prod_{p \in C}}'
K_p\otimes_{\oo_C}\ff\mbox{,}
$$
where $\prod'$ denotes the restricted (ad\`ele) product with respect
to $\prod\limits_{p \in C} \hat{\oo}_p$. Observe that since $\ff$ is
a subsheaf of $k(C)$, we have
$$
k(C)\otimes_{\oo_C}\ff=k(C),\qquad K_p\otimes_{\oo_C}\ff=K_p.
$$
The ad\`ele ring  $\da_C$ is a $1$-Tate vector space over $k$. This
is because
$$
\da_C =\mathop{\underleftarrow{\lim}}\limits_{\g \subset k(C)}
 \mathop{\underrightarrow{\lim}}\limits_{\h \subset k(C)} \da_{C,1}(\h) /
 \da_{C,1}(\g) \mbox{,}
$$
and $\dim_k \da_{C,1}(\h) /
 \da_{C,1}(\g) < \infty$ for coherent subsheaves $ 0 \ne \g \subset \h$ of
 $k(C)$. (We used that $\da_{C,1}(\h) /
 \da_{C,1}(\g) = \da_{C,1}= \bigoplus\limits_{p
\in C} \hat{\oo}_p \otimes_{\oo_C} (\h /\g )$). For any coherent
subsheaf $\ff$ of $ k(C)$ the space $\da_{C,1}(\ff)$
 is a lattice in the space $\da_C$.
Hence, we have that the $k$-space $k(C)$ is a colattice in $\da_C$,
since from the adleic complex $\ad(\ff)$ it follows
$$
\dim_k k(C) \cap \da_{C,1}(\ff) = \dim_k H^0(C, \ff) < \infty
\mbox{,}
$$
$$
 \dim_k \da_C / (k(C) + \da_{C,1}(\ff)) = \dim_k H^1(C, \ff) <
\infty \mbox{.}
$$

Let a $p$ be a point of $ C$ and   $f$, $g$  a pair of elements of
 $K_p^{\times}$. If $K_p = k(p)((t_p))$, then we denote by
$\{f,g\}_p$ the element from $k(p)^{\times}$ which is the
corresponding tame symbol. If the ring $K_p$ is isomorphic to the
finite product of fields isomorphic to $k(p)((t))$, then we denote
by $\{f,g\}_p$ the element from $k(p)^{\times}$ which is the same
finite product of the corresponding tame symbols. Recall that there
is the diagonal embedding $k(C) \hookrightarrow \da_C$.

\begin{prop}[Weil reciprocity law]
For any elements $f, g \in k(C)^{\times}$ the following product
contains only   finitely many nonequal to $1$ terms and
\begin{equation} \label{Weilform}
\prod_{p \in C}  \on{Nm}_{k(p)/k}\{f,g  \}_p  =1  \mbox{.}
\end{equation}
\end{prop}
\begin{proof} It is clear that, by proposition~\ref{tame symbol}, we can change  $\on{Nm}_{k(p)/k}\{f,g  \}_p \;$
to \\ $\Comm(\calD et_{K_p})(f,g)$ for all $p \in C$ in
formula~\eqref{Weilform}. There are points $p_1, \ldots, p_l \in C$
such that if $p \in C$ and $p \ne p_i$ $(1 \le i \le l)$, then $f
\hat{\oo}_p = \hat{\oo}_p$, $g \hat{\oo}_p = \hat{\oo}_p$, and
hence, by proposition~\ref{prop3.4}, $\Comm(\calD et_{K_p})(f,g) =1$
for  points $p \ne p_i$ $(1 \le i \le l)$.

We define the group $H$ as  the subgroup of the  group
$k(C)^{\times}$ generated by the elements $f$ and $g$. We apply
lemma~\ref{fil} to the following $1$-Tate $k$-vector spaces:
$$\calV = \da_C \mbox{, } \qquad \calV' = \da_{C
\setminus \{p_1, \ldots, p_l\}} \mbox{,}  \qquad \calV'' = \prod_{1
\le i \le l} K_{p_i} \mbox{.}$$ The group $H$ preserves the lattice
$\prod\limits_{p \in C \setminus \{p_1, \ldots, p_l  \}}
\hat{\oo}_p$ in the space $\calV'$. Therefore, by
proposition~\ref{prop3.4}, the homomorphism $\calD et_{\da_C}$ is
isomorphic to the homomorphism $\calD et_{\calV''}$, which is (again
by lemma~\ref{fil}) isomorphic to the sum  of homomorphisms $\calD
et_{K_{p_1}}, \ldots, \calD et_{K_{p_l}}$. Since the group $H$
preserves the colattice $k(C)$ in $\da_C$, the homomorphism $\calD
et_{\da_C}$ is isomorphic to the trivial one (by
proposition~\ref{prop3.4}). Now using remark~\ref{2.5} and
corollaries~\ref{Add}  and~\ref{2.8} we obtain
formula~\eqref{Weilform}.
\end{proof}
\begin{rmk}
To obtain the triviality of homomorphism~$\calD et_{\da_C} :
k(C)^{\times} \to \calP ic^{\bbZ}$ in an explicit way, one has to
use the following canonical isomorphism for any $g \in
k(C)^{\times}$:
\begin{equation} \label{detisom}
\calD et_{\da_C} (g) \: \simeq \: \det(H^* (\ad_C(g \oo_C))) \otimes
\det(H^* (\ad_C( \oo_C)))^{-1} \mbox{,}
\end{equation}
where for any coherent sheaf $\ff$ on $C$
\begin{multline*}
\det (H^*(\ad_C(\ff))) := \det (H^0 (\ad_C(\ff))) \otimes \det (H^1
 (\ad_C(\ff)))^{-1} \\ \simeq \det (H^0 (C, \ff) )\otimes \det (H^1
 (C, \ff))^{-1} \mbox{.}
\end{multline*}
(Formula~\eqref{detisom} easily follows from ad\`ele complexes and
formula~\eqref{grline} if we change in  formula~\eqref{grline} the
lattices $\calL$ and $g \calL$ in $\da_C$ to any two lattices coming
from nonzero coherent subsheaves $\g \subset \h$ of $k(C)$, and
change correspondingly in formula~\eqref{detisom} the sheaves $\oo$
and $g\oo$ to the sheaves $\g \subset \h$.) Now the  homomorphism
$\calD et_{\da_C}$ is isomorphic to the trivial one by
formula~\eqref{detisom} and the fact that multiplication on an
element $g \in k(C)^*$ gives a canonical isomorphism between ad\`ele
complexes $\ad_C(\oo_C)$ and $\ad_C(g\oo_C)$, which induce the
canonical isomorphism between $\det(H^* (\ad_C(\oo_C)))$ and
$\det(H^* (\ad_C( g\oo_C)))$.
\end{rmk}

\subsection{Parshin reciprocity laws} \label{sec5.2}
Let $X$ be an algebraic surface over the field $k$. We assume, for
simplicity, that $X$ is a smooth connected surface.

We consider pairs $x\in C$, where $C$ are irreducible curves on $X$ and $x$ are closed points on
$C$.  For every such pair one can define the ring
$K_{x,C}$, which will be a finite product of two-dimensional local fields, as follows. Assume that the curve
$C$ on $X$ has the following formal branches ${\mathbf C_1}, \ldots,
{\mathbf C_n}$ at the point $x \in C$, i.e.
$$C \mid_{\spec
\hat{\oo}_x} = \bigcup_{ 1 \le i \le n } {\mathbf C_i} \; \mbox{,}$$
where $\hat{\oo}_x$ is the completion of the local ring $\oo_x$ of a
point $x \in X$, and $ {\mathbf C_i}$ is irreducible in $\spec
\hat{\oo}_x$ for any $1 \le i \le n$. (Since we assumed $X$ is
smooth, $\hat{\oo}_x \simeq k(x)[[t_1,t_2]]$.) Now every $\mathbf
C_i$ defines a discrete valuation on the fraction field $\Frac
\hat{\oo}_x$. We define a two-dimensional local field $K_{x, \mathbf
C_i}$ as the completion of the field $\Frac \hat{\oo}_x$ with
respect to this discrete valuation, and let $\hat{\oo}_{x, \mathbf
C_i}$
 be the valuation ring. Then we define
$$
K_{x,C}: = \bigoplus\limits_{1 \le i \le n} K_{x, \mathbf C_i} \;
\mbox{,} \qquad
\hat{\oo}_{x,C}: = \bigoplus\limits_{1 \le i \le n} \hat{\oo}_{x,
\mathbf C_i} \; \mbox{,}
$$
Observe that if $x \in C$ is a smooth point, then $\hat{\oo}_{x,C} \simeq
k(x)((t))[[s]] $ and $ K_{x,C}\simeq k(x)((t))((s))$. It is clear that the ring
$\hat{\oo}_x$ diagonally embeds into the ring $K_{x,C}$.

Let us also define $B_x \subset K_{x,C} $ as
$\mathop{\underrightarrow{\lim}}\limits_{n > 0} s_C^{-n}
\hat{\oo}_x$, where a local equation $s_C =0$ determines $C$ on some
open $X \supset V \ni x$. It is clear that the subring $B_x$ of
$K_{x,C}$ does not depend on the choice of such $s_C$ when $V \ni
x$. If $x \in C$ is a smooth point, and $K_{x,C} =
k(x)((t))((s_C))$, where $s_C =0$ is a local equation of the curve
$C$ on $X$ near the point $x$ and $t=0$ defines a transversal curve
locally on $X$ near $x$, then $B_x = k(x)[[t]]((s_C))$.

Any ring $K_{x,C}$ is a $2$-Tate vector space over the field $k(x)$
(and therefore over the field $k$), and the ring $\hat{\oo}_{x,C}$ is a
lattice in $K_{x,C}$.

Let $f$, $g$, $h$ be from $K_{x,C}^{\times}$ such that $f =
\bigoplus\limits_{1 \le i \le n} f_i$, $g = \bigoplus\limits_{1 \le
i \le n} g_i$, $h = \bigoplus\limits_{1 \le i \le n} h_i$. Then we
define the following  element from $k(x)^{\times}$ as
\begin{equation}\label{around x,C}
\{f, g,h\}_{x,C} :=
\prod_{1 \le i \le n} \{f_i,g_i,h_i \}_{x, \mathbf C_i} \mbox{,}
\end{equation}
where $\{f_i,g_i,h_i \}_{x, \mathbf C_i}$ is the two-dimensional
tame symbol associated to the two-dimensional local field $K_{x,
\mathbf C_i}$ (cf. \S \ref{ts}).

\bigskip

Fix a point $x \in X$. For any free finitely generated
$\hat{\oo}_x$-module subsheaf
 $\ff$ of the constant sheaf $\Frac \hat{\oo}_x$ on the
scheme $\spec \hat{\oo}_x$ we consider the following ad\`ele complex
$\ad_{X,x}(\ff)$:
$$
\da_{X,x, 0}(\ff) \oplus \da_{X,x,1}(\ff) \lrto \da_{X,x, 01}(\ff)
\mbox{.}
$$
This is the ad\`ele complex on the $1$-dimensional scheme $U_x
:=\spec \hat{\oo}_x \setminus x$ for the sheaf $\ff \mid_{U_x}$,
and, hence, the cohomology groups of this complex coincide with the
cohomology groups $H^*(U_x,\ff \mid_{U_x} )$. By definition, we have
$$
\da_{X,x, 0} (\ff) = \Frac \hat{\oo}_x \mbox{,} \quad \da_{X,x, 1}
(\ff)=\prod_{{\mathbf C} \ni x} \hat{\oo}_{x,{\mathbf C}}
\otimes_{\hat{\oo}_x} \ff \mbox{,} \quad \da_{X,x, 01}(\ff) =
\da_{X,x} = {\prod_{ {\mathbf C} \ni x  }}' K_{p, {\mathbf C }}
\mbox{,}
$$
where the product is taken over all prime ideals $\mathbf C$ of
height $1$ of the ring $\hat{\oo}_x$, and  $\prod'$ denotes the
restricted (ad\`ele) product with respect to $\prod\limits_{\mathbf
C \ni x} \hat{\oo}_{x, \mathbf C }$.

Observe that the ad\`ele ring  $\da_{X, x}$ is a $2$-Tate vector
space over the field $k(x)$. This is because
$$
\da_{X,x} =\mathop{\underleftarrow{\lim}}\limits_{\g \subset \Frac
\hat{\oo}_x}
 \mathop{\underrightarrow{\lim}}\limits_{\h \subset \Frac \hat{\oo}_x} \da_{X,x,1}(\h) /
 \da_{X,x,1}(\g) \mbox{,}
$$
and $ \da_{X,x,1}(\h) /
 \da_{X,x,1}(\g) $ is a $1$-Tate vector space for free $\hat{\oo}_x$-module subsheaves $ 0 \ne \g \subset \h$ of
 $\Frac \hat{\oo_x}$. (We used that $ \da_{X,x,1}(\h) /
 \da_{X,x,1}(\g) = \bigoplus\limits_{\mathbf C \ni x} \hat{\oo}_{x,{\mathbf C}} \otimes_{\hat{\oo}_x} \h /\g$.)
 For any  free finitely generated $\hat{\oo}_x$-module subsheaf $\ff$ of $\Frac{\hat{\oo_x}}$ the space
 $\da_{X,x,1}(\ff)$
 is a lattice in the space $\da_{X,x}$.

 From~\cite[prop.~8]{O3} it follows that $k(x)$-vector spaces $H^0
 (\ad_{X,x}(\ff))$ and $H^1(\ad_{X,x}(\ff))$ are $1$-Tate vector
 spaces. Indeed, since $x$ is a smooth point of $X$,
 $$
H^0 (\ad_{X,x}(\ff)) = H^0(U_x, \ff \mid_{U_x}) = \ff
 $$
 is a projective limit of finite-dimensional $k(x)$-vector
 spaces $\ff / m_x^n \ff$ ($m_x$ is the maximal ideal of the ring $\hat{\oo}_x$),
 and
 $$
H^1 (\ad_{X,x}(\ff)) = H^1(U_x, \ff \mid_{U_x}) =
\mathop{\underrightarrow{\lim}}\limits_{n > 0} \Ext^2_{\hat{\oo}_x}
(\hat{\oo}_x / m_x^n, \ff) \mbox{,}
 $$
 where for any $n > 0$ the space $\Ext^2_{\hat{\oo}_x}
(\hat{\oo}_x / m_x^n, \ff)$ is a finite-dimensional over the field
$k(x)$ vector space (see, for example,~\cite[lemma~6]{O3}).

\bigskip

Fix an irreducible projective curve $C$ on $X$. For any invertible
$\oo_X$-subsheaf $\ff$ of the constant sheaf $k(X)$ on $X$ we
consider the following ad\`ele complex $\ad_{X,C}(\ff)$
$$
\da_{X,C, 0}(\ff) \oplus \da_{X,C,1}(\ff) \lrto \da_{X,C, 01}(\ff)
\mbox{.}
$$
where $$ \da_{X,C, 0}(\ff): = K_C, \quad \da_{X,C,
01}(\ff):=\da_{X,C}= \da_C((s_C)) \mbox{,}
$$
\begin{equation} \label{intersect}
 \da_{X,C,1}(\ff): =
(\prod\limits_{x \in  C} B_x \otimes_{\oo_X} \ff)  \cap \da_{X,C}
\mbox{.}
\end{equation}
 Here $K_C$ is the completion of the field $k(X)$ with
respect to the discrete valuation given by the curve $C$ on $X$. (If
$s_C = 0$ is a local equation of the curve $C$ on some open subset
$V$ of $X$ such that $V \cap C \ne \emptyset$, then $K_C =
k(C)((s_C))$.) The ring $\da_{X,C}$ is a subring of $\prod\limits_{x
\in C} K_{x,C}$, and does not depend on the choice of $s_C$. The
intersection~\eqref{intersect} is taken in the ring $\prod\limits_{x
\in C} K_{x,C}$.

We note that  from~\cite[\S~5.1]{O3} it follows
that the complex $\ad_{X,C}(\ff)$ coincides with the following
complex
$$\mathop{\underrightarrow{\lim}}\limits_n \mathop{\underleftarrow{\lim}}\limits_{m
> n} \ad_{(C, \oo_X/J_C^{m-n})} (\ff \otimes_{\oo_X} J_C^n / J_C^m)  \mbox{.}
$$
Here $J_C$ is the ideal sheaf of the curve $C$  on $X$, $(C,
\oo_X/J_C^{m-n})$ is a $1$-dimensional scheme which has the
topological space  $ C$ and the structure sheaf $\oo_X/J_C^{m-n}$,
and $\ad_{(C, \oo_X/J_C^{m-n})} (\ff \otimes_{\oo_X} J_C^n / J_C^m)$
is the ad\`ele complex of the coherent sheaf  $\ff \otimes_{\oo_X}
J_C^n / J_C^m$ on the scheme $(C, \oo_X/J_C^{m-n})$.  Hence and from
the proof of proposition~12 from~\cite{O3} we obtain that
$$
H^*(\ad_{X,C}(\ff)) = \mathop{\underrightarrow{\lim}}\limits_n
\mathop{\underleftarrow{\lim}}\limits_{m
> n} H^*(C, \ff \otimes_{\oo_X} J_C^n
/ J_C^m)  \mbox{,}
$$
where for $i=0$ and $i=1$ we have $\dim_k H^i(C, \ff \otimes_{\oo_X}
J_C^n / J_C^m ) < \infty$. For $i=0$ and $i=1$ the $k$-vector space
$H^i(\ad_{X,C}(\ff))$ has the natural topology of inductive and
projective limits. It is not difficult to see that the space
$H^0(\ad_{X,C}(\ff))$ is a locally linearly compact $k$-vector
space, i.e. it is a $1$-Tate vector space. But the space
$H^1(\ad_{X,C}(\ff))$ is not a Hausdorff space in this topology. Let
$\tilde{H}^1(\ad_{X,C}(\ff))$ be the quotient space of
$H^1(\ad_{X,C}(\ff))$ by the closure of zero. Then the space
$\tilde{H}^1(\ad_{X,C}(\ff))$ is a locally linearly compact
$k$-vector space, i.e. a $1$-Tate vector space.

We note that for any invertible subsheaves $ 0 \ne \g \subset \h$ of
$k(X)$ we have that the space $B_x \otimes_{\hat{\oo}_x} (\h/\g)$ is
a $1$-Tate vector space, which is equal to zero for almost all
points $x \in C$. Hence, we obtain that the space $$\da_{X,C,1}(\h)
/ \da_{X,C,1}(\g) = \bigoplus\limits_{x \in C } B_x
\otimes_{\hat{\oo}_x} (\h/\g)$$ is a $1$-Tate vector space.

\medskip

For any point $x \in X$, we define a ring $K_x$ as the localization
of the ring $\hat{\oo}_x$ with respect to the multiplicative system
$\oo_x \setminus 0$. (We note that inside of the field $\Frac
\hat{\oo}_x$ the ring $K_x$ is defined as the product of two
subrings: $\hat{\oo}_x$ and $k(X)$.)

For any pair $x \in C$ (where $C$ is an irreducible curve on $X$ and
$x\in C$ is a closed point), we have the natural embeddings $k(X)
\hookrightarrow K_x$, $k(X)\hookrightarrow K_C$ (recall that $K_C$
is the completion of the field $k(X)$ with respect to the discrete
valuation given by the curve $C$). In addition, there are the
natural embeddings $K_x, K_C \hookrightarrow K_{x,C}$. Therefore, we
obtain
$$k(X) \hookrightarrow K_x
\hookrightarrow \da_{X,x} \qquad \mbox{,} \qquad k(X)
\hookrightarrow K_C \hookrightarrow \da_{X,C} \mbox{.}$$

\begin{thm}[Parshin reciprocity laws]
\begin{enumerate}
\item
 Fix a point $x \in X$. Consider elements $f,g,h$  of the group $K_x^{\times}$ of invertible elements of
 the ring $K_x$.
 Then the following product in $k(x)^{\times}$
contains only  finitely many non-equal to $1$ terms and
\begin{equation} \label{Parshin1}
\prod_{C \ni x} \{f,g,h \}_{x,C} = 1  \mbox{.}
\end{equation}
\item Fix a projective irreducible curve $C$ on $X$. Let elements $f,g,h$ be from the group  $K_C^{\times}$.
Then the following product in $k^{\times}$
contains only  finitely many non-equal to $1$ terms and
\begin{equation}  \label{Parshin2}
\prod_{x \in C} \on{Nm}_{k(x)/k} \{f,g,h \}_{x,C} = 1  \mbox{.}
\end{equation}
\end{enumerate}
\end{thm}
\begin{proof} We first prove formula~\eqref{Parshin1}. By
theorem~\ref{2-tame}, for any $f,g,h \in K_{x,\mathbf C}^{\times}$
we have
\begin{equation} \label{fffor}
\{ f,g,h\}_{x, \mathbf C} = C_3^{\calD et_{x,\mathbf C}}(f,g,h)
\end{equation}
 for
all prime ideals $\mathbf C$ of height $1$ of the ring
$\hat{\oo}_{x}$, where the central extension $\calD et_{x,\mathbf
C}$ of the group $K_{x,\mathbf C}^{\times}$ by the Picard groupoid
$\calP ic^{\bbZ}$ is constructed by formula~\eqref{pictorsor} from
the $2$-Tate vector space $K_{x, \mathbf C}$ over the field $k(x)$
and the lattice $\hat{\oo}_{x,\mathbf C}$ as in
section~\ref{2-story}. We note that for almost all  prime ideals $
\mathbf C$ of height $1$ of ring $\hat{\oo}_{x}$, and for any
elements $f,g,h$ from the group $\Frac \hat{\oo}_x^{\times}$, we
have $f \oo_{x,\mathbf C}= \oo_{x, \mathbf C}$, $g \oo_{x, \mathbf
C}= \oo_{x,\mathbf C}$, and $h \oo_{x, \mathbf C}= \oo_{x, \mathbf
C}$. Then by lemma~\ref{t} and corollary~\ref{trivial}, for almost
all prime ideals $\mathbf C$ of height $1$ of ring $\hat{\oo}_x$ we
have $C_3^{\calD et_{x,\mathbf C}}(f,g,h)=1$.

We will prove that the central extension $\calD et_{x}$ of $\Frac
\hat{\oo}_x^\times$($\subset\GL(\da_{X,x})$) by $\calP ic^{\bbZ}$
constructed by the $2$-Tate vector space $\da_{X,x}$ and the lattice
$\da_{X,x,1}(\hat{\oo}_x)$ using formula~\eqref{pictorsor} can be
trivialized in an explicit way. Observe that for any $d \in
\Frac\hat{\oo}_x^{\times}$, there is a canonical isomorphism of
$\calP ic^{\bbZ}$-torsors:
\begin{equation} \label{detisom2}
\calD et (\da_{X,x,1}( \hat{\oo}_x)) \mid \da_{X,x,1}(d
\hat{\oo}_x)) \: \simeq \: \calD et (H^* (\ad_{X,x}(d \hat{\oo}_x)))
- \calD et(H^* (\ad_{X,x}( \hat{\oo}_x))) \mbox{,}
\end{equation}
where for any free subsheaf $\ff$ of $\Frac \hat{\oo}_x$ on the
scheme $\spec \hat{\oo_x}$
$$
\calD et (H^*(\ad_{X,x}(\ff))) := \calD et (H^0 (\ad_{X,x}(\ff))) -
\calD et (H^1
 (\ad_{X,x}(\ff)))  \mbox{.}
$$
Indeed, isomorphism~\eqref{detisom2} follows from
proposition~\ref{det gerbe} applied to the long exact sequence
(decomposed into the short exact sequences) associated with the
following exact sequence of complexes of length  $2$ for any nonzero
free subsheaves $\g \subset \h$ of $\Frac \hat{\oo}_x$ on the scheme
$\spec \hat{\oo_x}$:
$$
0 \lrto \ad_{X,x}(\g) \lrto \ad_{X,x}(\h) \lrto \ad_{X,x,1}(\h)/
\da_{X,x,1}(\g) \lrto 0 \mbox{,}
$$
where the last complex consists only of the group placed  in degree
zero. Now we have
\begin{multline} \label{mult2}
\calD et (H^* (\ad_{X,x}(d \hat{\oo}_x))) - \calD et(H^* (\ad_{X,x}(
\hat{\oo}_x)))  \\ \simeq \Hom_{\calP ic^{\bbZ}} (\calD et
(H^* (\ad_{X,x}(\hat{\oo}_x))), \calD et(H^* (\ad_{X,x}(d
\hat{\oo}_x))) \mbox{.}
\end{multline}
The multiplication by the element $d \in
\Frac{\hat{\oo}_x}^{\times}$ between ad\`ele complexes $\ad_{X,x}(
\hat{\oo}_x)$ and $\ad_{X,x}(d \hat{\oo}_x)$ gives a natural
isomorphism of $\calP ic^{\bbZ}$-torsor from formula~\eqref{mult2}
to the trivial torsor $\calP ic^{\bbZ}$.

Let $H$ be the subgroup of $\Frac \hat{\oo}_x^{\times}$ generated by
the elements $f,g,h \in \Frac \hat{\oo}_x^{\times}$. Now we proceed
as the proof of Weil reciprocity law (see
proposition~\ref{Weilform}), with the help of lemma~\ref{fil2},
lemma \ref{t}, and corollary~\ref{2Baer}. Then we obtain the
following equality:
$$
\prod_{ {\mathbf C} \ni x} \{f,g,h  \}_{x, \mathbf C} =1  \mbox{.}
$$

Formula~\eqref{Parshin1} follows from the last formula, since if a
prime ideal $\mathbf C$ of height $1$ in $\hat{\oo}_x$ is not a
formal branch at $x$ of some irreducible curve $C$ on $X$, then for
any element $d \in K_x^\times$ we have $d \oo_{x, \mathbf C} =
\oo_{x, \mathbf C}$. Hence, by formula~\eqref{fffor}, $\{f,g,h
\}_{x, \mathbf C}=1$ for such $\mathbf C$ and any $f,g,h \in
K_x^{\times}$.

\medskip

Next we will prove formula~\eqref{Parshin2}. We construct the
central extension $\calD et'_{x,C}$ of the group $k(X)^{\times}$  by
the Picard groupoid $\calP ic^{\bbZ}$ in the following way. We fix a
point $x \in C$, and associate with the rings $B_x \subset K_{x, C}$
and with an element $d \in k(X)^{\times}$ the following $\calP
ic^{\bbZ}$-torsor:
\begin{equation} \label{detlast}
\Det(B_x \mid d B_x):=\Det \left(\frac{d B_x}{B_x\cap d B_x} \right)
-\Det \left(\frac{B_x}{B_x\cap d B_x} \right) \mbox{.}
\end{equation}
(We used that $B_x / B_x \cap d B_x$ is a $1$-Tate vector space over
the field $k$.) By the formula which is analogous to
formula~\eqref{deiso} we obtain that the central extension $\calD
et'_{x,C}$ is well defined. In a similar way we define the central
extensions $\calD et'_{C}$ and $\calD et'_{C \setminus \{x_1, \ldots
, x_l \}}$ starting from the rings $\da_{X,C,1}(\oo_X) \subset
\da_{X,C}$ and $\da_{X,C \setminus\{ x_1, \ldots, x_l\},1}(\oo_X)
\subset \da_{X,C \setminus \{x_1, \ldots, x_l \}}$, where $x_1,
\ldots, x_l$ are some points on the curve $C$.

Let the group $H$ be generated in the group $k(X)^{\times}$ by the
elements $f,g,h \in k(X)^{\times}$.
 For almost all points $x$ of the curve $C$
we have that the group $H$ preserves the subring $B_x$. Therefore
form formula~\eqref{detlast} we obtain that the central extension
$\calD et'_{x,C}$ is isomorphic to the trivial one for almost all
points $x$ of the curve $C$. Therefore for almost all points $x$ of
the curve $C$ we have $C_3^{\calD et'_{x, C}}(f,g,h) =1$.

We will prove that the central extension $\calD et'_{x,C}$ is
inverse (or dual) to the central extension $\calD et_{x,C}$, where
the last central extension is constructed by
formula~\eqref{pictorsor} from the lattice $\oo_{x,C}$ in the
$2$-Tate vector space $K_{x,C}$. For any free subsheaf  $\ff$ of the
constant sheaf $\Frac \hat{\oo}_x$ on the scheme $\spec \hat{\oo}_x$
there is the following complex $\ad_{X,C,x}(\ff)$:
$$
(B_x \otimes_{\hat{\oo}_x} \ff) \; \oplus \; (\hat{\oo}_{x,C}
\otimes_{\hat{\oo}_x} \ff) \lrto K_{x,C} \mbox{.}
$$
We have canonically that $H^*(\ad_{X,C,x}(\ff)) = H^* (U_x, \ff
\mid_{U_x})$, where we recall $U_x = \spec \hat{\oo}_x \setminus x $
(see the proof of proposition~13 from~\cite{O3}). Therefore the
cohomology groups of complex $\ad_{X,C,x}(\ff)$ are $1$-Tate vector
spaces. Hence, there is a canonical isomorphism between the
following  $\calP ic^{\bbZ}$-torsors for any $d \in k(X)^{\times}$:
$$
\Det(B_x \mid d B_x) + \Det(\hat{\oo}_{x,C} \mid d \hat{\oo}_{x,C})
\qquad \mbox{and}
$$
$$
\Hom_{\calP ic^{\bbZ}} (\calD et(H^*(\ad_{X,C,x}(\hat{\oo}_x))) \,,
\, \calD et(H^*(\ad_{X,C,x}(d\hat{\oo}_x)))   ) \mbox{.}
$$
Now the multiplication by the element $d$ of ad\`ele complexes gives
a natural isomorphism from the last $\calP ic^{\bbZ}$-torsor to the
trivial one. Hence from corollary~\ref{2Baer} we have that
$$
C_3^{\calD et'_{x,C}} (f,g,h) = {{C_3^{\calD et'_{x,C}}}(f,g,h)}
^{-1} = \on{Nm}_{k(x)/k} \{f,g,h \}_{x,C}^{-1}
$$
for $f,g,h \in k(X)^{\times}$.

Now the proof of formula~\eqref{Parshin2} for elements $f,g,h \in
k(X)^{\times}$ follows by the same method as in the proof of
formula~\eqref{Parshin1}, but we have to use the ad\`ele ring
$\da_{X,C}$ instead of the ring $\da_{X,x}$, and to use the central
extension $\Det'_{C}$ instead of the central extension $\Det_{{x}}$.
We need only to prove that the central extension $\calD et'_{C}$
constructed by the analog of formula~\eqref{detlast} from the rings
$\da_{X,C,1}(\oo_X) \subset \da_{X,C}$ is isomorphic the trivial
central extension. This follows if we consider the following $\calP
ic^{\bbZ}$-torsors for $d \in k(X)^{\times}$
\begin{equation} \label{ddd}
\Hom_{\calP ic^{\bbZ}} (\calD et(H^*(\ad_{X,C}(\oo_X))) \,, \, \calD
et(H^*(\ad_{X,C}(d \oo_X))) ) \mbox{,}
\end{equation}
 where
$$
\calD et(H^*(\ad_{X,C}(d \oo_X)) := \calD et(H^0(\ad_{X,C}(d \oo_X))
- \calD et(\tilde{H}^1(\ad_{X,C}(d \oo_X)) \mbox{.}
$$
Multiplication by $d \in k(X)^{\times}$ of ad\`ele complexes  gives
the triviality of the $\calP ic^{\bbZ}$-torsor~\eqref{ddd}. (See
analogous reasonings earlier in the proof of this theorem.)

To obtain formula~\eqref{Parshin2} for elements $f,g,h \in
K_C^{\times}$ we have to use that the field $k(X)$ is dense in the
field $K_{C}$. Therefore  for any element $f\in K_C^{\times}$ there
is an element $\tilde{f} \in k(X)^{\times}$ such that $f = \tilde{f}
m$, where the element $m$ is from the subgroup $ 1 + m_C^n$ of the
group $K_C^{\times}$ for some $n \ge 1$, and $m_C$ is the maximal
ideal of the valuation ring of discrete valuation field $K_C$. Then
from formula~\eqref{ts3} we have that $\{m, g, h\}_{x, \mathbf C}
=1$ for any point $x \in C$, and any formal branch $\mathbf C$ of
the curve $C$ at point $x$. Hence, from the tri-multiplicativity of
the two-dimensional tame symbol we obtain that
$$
\{f,g,h\}_{x,\mathbf C} = \{\tilde{f},g,h\}_{x,\mathbf C} \mbox{.}
$$
Applying successively   the same procedure to elements $g,h \in
k_C^{\times}$ we obtain
$$
\{f,g,h\}_{x,\mathbf C} = \{\tilde{f},\tilde{g},
\tilde{h}\}_{x,\mathbf C} \mbox{,}
$$
where $\tilde{f}, \tilde{g}, \tilde{h} \in k(X)^{\times}$, and any
point $x \in C$, and $\mathbf C$ is any formal branch of the curve
$C$ at point $x$.

\end{proof}

\begin{rmk}
For the proof of Parshin reciprocity laws we used "semilocal"
ad\`ele complexes of length $2$ connected with either points or
irreducible curves on an algebraic surface. But for the formulation
of these reciprocity laws we used the rings $K_x$ and $K_C$ which
appear from the "global" ad\`ele complex of length $3$ on an
algebraic surface. It would be interesting to find direct
connections between the "global" ad\`ele complex and "semilocal"
ad\`ele complexes of an algebrac surface.
\end{rmk}

\begin{rmk}
We have a symmetric monoidal functor from the Picard torsor $\calP
ic^{\bbZ}$ to the Picard groupoid $\bbZ$ which sends every graded
line to its grading element from $\bbZ$,  where $\bbZ$ is considered
as the  groupoid with objects equal to $\bbZ$ and morphisms equal to
identities morphisms. Under this functor a central extension of a
group $G$ by a $\calP ic^{\bbZ}$-torsor goes to the usual central of
the group $G$ by the group $\bbZ$. In this way the map $\nu_{\bbK}$
for a two-dimensional local field $\bbK$ was obtained as the
commutator of elements in this central extension in~\cite{O3}. Also
in~\cite{O3} the reciprocity laws for the map $\nu_{\bbK}$ were
proved by the ad\`ele complexes on an algebraic surface.
\end{rmk}

\bigskip

\noindent Denis Osipov, \\
Steklov Mathematical Institute, Gubkina
str. 8,
Moscow, 119991, Russia \\
E-mail address:  ${d}_{-} osipov@mi.ras.ru$

\bigskip
\noindent Xinwen Zhu, \\
Department of Mathematics, Harvard University, one Oxford Street, Cambridge, MA 02138, USA \\
E-mail address:  $xinwenz@math.hatvard.edu$

\end{document}